\documentclass[12pt,reqno]{amsart}
\usepackage{amscd,amsmath,amsthm,amssymb}
\usepackage[left]{lineno}
\usepackage{color}
\usepackage{stmaryrd}
\usepackage[utf8]{inputenc}
\usepackage{cleveref}
\usepackage{epstopdf}
\usepackage{graphicx}
\usepackage{xcolor}
\usepackage{subeqnarray}

\usepackage{cases}
\usepackage{subcaption}

\usepackage{tikz}
\usetikzlibrary{shapes,positioning}
\usepackage{graphicx}

\definecolor{verylight}{gray}{0.97}
\definecolor{light}{gray}{0.9}
\definecolor{medium}{gray}{0.85}
\definecolor{dark}{gray}{0.6}

\def\NZQ{\mathbb}               % the font for N,Z,Q,R,C

\def\KK{{\NZQ K}}

\def\KK{{\NZQ K}}

\def\G{{\mathcal G}}

\def\0b{{\mathbf 0}}
\def\alphab{{\mathbf \alphab}}
\def\c_ib{{\mathbf c_i}}

\def\reg{{\mathbf reg}}
\def\height{\operatorname{ht}}
\def\depth{\operatorname{depth}}
\def\opn#1#2{\def#1{\operatorname{#2}}} % to make operators
\opn\chara{char} \opn\length{\ell} \opn\pd{pd} \opn\rk{rk}
\opn\projdim{proj\,dim} \opn\injdim{inj\,dim} \opn\rank{rank}
\opn\depth{depth} \opn\grade{grade} \opn\height{height}
\opn\embdim{emb\,dim} \opn\codim{codim}

\opn\Tr{Tr} \opn\bigrank{big\,rank}
\opn\superheight{superheight}\opn\lcm{lcm}
\opn\trdeg{tr\,deg}%\emph{
	\opn\reg{reg} \opn\lreg{lreg} \opn\ini{in} \opn\lpd{lpd}
	\opn\size{size} \opn\sdepth{sdepth}
	\opn\link{link}\opn\fdepth{fdepth}\opn\lex{lex}
	\opn\tr{tr}
	\opn\type{type}
	\opn\gap{gap}
	\opn\arithdeg{arith-deg}
	\opn\HS{HS}
	\opn\GL{GL}
	\opn\div{div} \opn\Div{Div} \opn\cl{cl} \opn\Cl{Cl}
	\opn\Spec{Spec} \opn\Supp{Supp} \opn\supp{supp} \opn\Sing{Sing}
	\opn\Ass{Ass} \opn\Min{Min}\opn\Mon{Mon}
	\opn\Ann{Ann} \opn\Rad{Rad} \opn\Soc{Soc}\opn\Deg{Deg}
	\opn\Im{Im} \opn\Ker{Ker} \opn\Coker{Coker} \opn\Am{Am}
	\opn\Hom{Hom} \opn\Tor{Tor} \opn\Ext{Ext} \opn\End{End}
	\opn\Aut{Aut} \opn\id{id}
	
	\opn\nat{nat}
	\opn\pff{pf}%   \pf exists already
	\opn\Pf{Pf} \opn\GL{GL} \opn\SL{SL} \opn\mod{mod} \opn\ord{ord}
	\opn\Gin{Gin} \opn\Hilb{Hilb}\opn\sort{sort}
	\opn\PF{PF}\opn\Ap{Ap}
	\opn\mult{mult}
	\opn\bight{bight}
	\opn\Facets{Facets}
	\opn\aff{aff}
	\opn\relint{relint} \opn\st{st}
	\opn\lk{lk} \opn\cn{cn} \opn\core{core} \opn\vol{vol}  \opn\inp{inp} \opn\nilpot{nilpot}
	\opn\link{link} \opn\star{star}\opn\lex{lex}\opn\set{set}
	\opn\width{wd}
	\opn\Fr{F}
	\opn\QF{QF}
	\opn\G{G}
	\opn\type{type}\opn\res{res}
	\opn\conv{conv}
	\opn\Ind{Ind}
	\opn\gr{gr}
	
	\def\pot#1#2{#1[\kern-0.28ex[#2]\kern-0.28ex]}

	\opn\dirlim{\underrightarrow{\lim}}
	\opn\inivlim{\underleftarrow{\lim}}
	\let\to=\rightarrow
	
	\def\Implies{\ifmmode\Longrightarrow \else
		\unskip${}\Longrightarrow{}$\ignorespaces\fi}
	\def\implies{\ifmmode\Rightarrow \else
		\unskip${}\Rightarrow{}$\ignorespaces\fi}
	\def\iff{\ifmmode\Longleftrightarrow \else
		\unskip${}\Longleftrightarrow{}$\ignorespaces\fi}

	\let\:=\colon
	\newtheorem{Theorem}{Theorem}[section]
	\newtheorem{Lemma}[Theorem]{Lemma}

	\newtheorem{Example}[Theorem]{Example}
	
	\newtheorem{Definition}[Theorem]{Definition}

	\let\epsilon\varepsilon
	\let\kappa=\varkappa
	\opn\dis{dis}
	\def\pnt{{\raise0.5mm\hbox{\large\bf.}}}
	
	\opn\Lex{Lex}
	\opn\vstab{vstab}
	
	\newcommand*{\circled}[1]{\lower.7ex\hbox{\tikz\draw (0pt, 0pt)%
			circle (.5em) node {\makebox[1em][c]{\small #1}};}}

\begin{document}

		%\linenumbers
		\title {The  $\mathrm{v}$-number of powers of ideals associated with chessboard complexes}
		
\author {Wenjie Huo, Guangjun Zhu$^{\ast}$ and Jiaxin Li}
\address{ School of Mathematical Sciences, Soochow University, Suzhou 215006, P. R. China}
	
	\email{wjhuo021030@stu.suda.edu.cn(Wenjie  Huo), zhuguangjun@suda.edu.cn\\
		(Corresponding author:Guangjun Zhu),
	lijiaxinworking@163.com(Jiaxin Li).}

		\thanks{$^{\ast}$ Corresponding author}
		
		\thanks{2020 {\em Mathematics Subject Classification}.
			Primary 13F20,  13F55; Secondary 05E40}

		\thanks{Keywords:  $\mathrm{v}$-number, symbolic power,  ordinary	 power, facet ideal, Stanley--Reisner ideal, chessboard complex}

		% \subjclass[2010]{Primary 13C99; Secondary 13E15, 13A15.}
		%		13H10   	Special types (Cohen-Macaulay, Gorenstein, Buchsbaum, etc.)
		%		13D02   	Syzygies, resolutions, complexes
		%		05E40   	Combinatorial aspects of commutative algebra
		%		16S36   	Ordinary and skew polynomial rings and semigroup rings
		
		%		14M25   	Toric varieties, Newton polyhedra [See also 52B20]
		%		13A02   	Graded rings
		%		13F20   	Polynomial rings and ideals; rings of integer-valued polynomials
		%		13A18   	Valuations and their generalizations
		%		06A11   	Algebraic aspects of posets
		%       05C38       Paths and cycles
		%       13A15       Ideals; multiplicative ideal theory
		%       13F20       Polynomial rings and ideals; rings of integer-valued polynomials
		%   13F20 Polynomial rings and ideals; rings of integer-valued polynomials [See also 11C08, 13B25]
		
		%\keywords{   }

\maketitle

		\begin{abstract}
		We study the $\mathrm{v}$-numbers of ordinary and symbolic powers of the facet ideal and the Stanley--Reisner ideal of a chessboard complex. Exact formulas are established for the $\mathrm{v}$-numbers of ordinary and symbolic powers of the Stanley--Reisner ideal, as well as symbolic powers of the facet ideal. Furthermore, we show that the $\mathrm{v}$-number function of the ordinary powers is asymptotically linear.
		\end{abstract}
		\setcounter{tocdepth}{1}
		%\tableofcontents

		\section{Introduction}
		
		Let $S=\KK[x_1,\dots, x_n]=\bigoplus\limits_{d\ge 0} S_d$ be the standard graded polynomial ring  over a field $\KK$.
	For a graded ideal $I \subseteq S$, let $\operatorname{Ass}(I)$ denote the set of all associated prime ideals of $I$. 
	 For  each $\frak{p}\in \operatorname{Ass}(I)$,
		the \emph{local $\mathrm{v}$-number} of $I$ at $\frak{p}$ 
		is defined as
		\[
	\mathrm{v}_{\mathfrak{p}}(I)= \min\{d\ge 0 \mid \text{\ there exists \ } f\in S_d \text{\ such that\ } (I : f)=\mathfrak{p}\}.
		\]
		The \emph{$\mathrm{v}$-number} 
		of $I$, denoted by $\mathrm{v}(I)$, is then defined as
		\[
		\mathrm{v}(I)= \min\{\mathrm{v}_{\mathfrak{p}}(I)\mid \mathfrak{p}\in \operatorname{Ass}(I)\}.
		\]
		The $\mathrm{v}$-number was introduced in \cite{CSTVV} to analyze  the asymptotic behavior of the minimum distance of projective Reed-Muller type codes. This invariant  has been extensively studied from both algebraic and combinatorial perspectives. Much of this research focuses on monomial ideals (see, e.g., \cite{BM,  F, KMT, SS}). In particular, Jaramillo and Villarreal \cite{JV} described $\mathrm{v}(I)$ combinatorially for edge ideals of clutters, while Kataoka et al. \cite{KMT} studied the $\mathrm{v}$-number of Stanley--Reisner ideals via Alexander duality, proving that $\mathrm{v}(I_\Delta)$ equals the dimension of $S/I_\Delta$ whenever $\Delta$ is a $2$-pure simplicial complex.
		
		Regarding the asymptotic behavior of $\mathrm{v}$-numbers, Ficarra and Sgroi \cite{FS} and Conca \cite{Conca} independently established that $\mathrm{v}(I^t)$ is eventually linear in $t$ for any graded ideal $I$ in a polynomial ring (and more generally, in an $\mathbb{N}$-graded Noetherian domain \cite{Conca}).  We refer the reader to \cite{BMS, F, FS, FG,KNS,VS} for further developments in this direction.	In particular,  for any Noetherian graded filtration $\mathcal{I} = \{I_{t}\}$ of an $\mathbb{N}$-graded Noetherian domain $R$, it has been proved that $\mathrm{v}(I_{t})$ is eventually a quasi-linear function in $t$, i.e., $\mathrm{v}(I_{t})$ is a periodically linear function in $t$ for $t \gg 0$.
		
		Comparing ordinary and symbolic powers $I^t$ and $I^{(t)}$ is a central theme in the study of graded ideals. For squarefree monomial ideals $I$, Chau et al. \cite{CHJV} recently investigated the relation between the $\mathrm{v}$-numbers of these two filtrations, proving that
		\[
		\mathrm{v}(I^t) \le \mathrm{v}(I^{(t)}) + (t-1)(d-1) \quad \text{for all } t \ge 1,
		\]
		where $d$ denotes the maximum degree of a minimal generator of $I$. They further showed that for any fixed $t \ge 2$, the difference $\mathrm{v}(I^{(t)}) - \mathrm{v}(I^t)$ can be arbitrarily large. More recently, Kumar et al. \cite{KMV} studied the $\mathrm{v}$-numbers of two associated discrepancy modules, $I^{(t)}/I^t$ and $I^{(t)}/I^{(t+1)}$.
	
		For the edge ideal   of a simple graph, the asymptotic behavior of its $\mathrm{v}$-number is completely determined by the results of Biswas  et al.   \cite{BMS} and Ficarra  et al. in  \cite{FM}. Moreover, Kumar et al. \cite{KNS} proved that
	\[
	\lim_{t \to \infty} \frac{\mathrm{v}(I^{(t)})}{t} = \widehat{\alpha}(I),
	\]
	where $\widehat{\alpha}(I)$ denotes the Waldschmidt constant of $I$.
	
	Chessboard complexes, which are the matching complexes $\mathcal{M}(K_{m,n})$	of the complete bipartite graphs $K_{m,n}$,  were first introduced in the thesis of Garst \cite{Garst}. He showed that $\mathcal{M}(K_{m,n})$ is Cohen--Macaulay if and only if $n\geq 2m-1$.
	We refer the reader to \cite{BLVZ,FH,Jojic ,Zi}  for further developments in this direction.
	The chessboard complex $\Delta_{m,n}$ is formed by all admissible rook configurations on an $m \times n$ chessboard, where an admissible configuration is a subset of squares with no two in the same row or column.
	
	In this paper, we focus on the $\mathrm{v}$-numbers of ordinary and symbolic powers of the  facet ideal and the Stanley--Reisner ideal of   a chessboard complex. 	Our main results are as follows:
	
	\begin{Theorem}
		Let	$n \geq m\ge 1$ be   integers.	Let $\Delta_{m,n}$ be a chessboard complex and $I_{\Delta_{m,n}}$ its Stanley-Reisner ideal. 
		Then
		\[
		\mathrm{v}(I_{\Delta_{m,n}})=
		\begin{cases}
			m-1, &\text{if $m=n$},\\
			m, &\text{if $n>m$}.
		\end{cases}
		\]
	\end{Theorem}

	\begin{Theorem}
		Let  $n \geq m\ge 1$ and $t\ge 2$ be  integers. Let $\Delta_{m,n}$ be a chessboard complex and $I_{\Delta_{m,n}}$ its Stanley-Reisner ideal. Then 
		\[
		\mathrm{v}(I_{\Delta_{m,n}}^t)=	\begin{cases}
			0, &\text{if $m=n=1$},\\
			m, &\text{if $n \geq m \geq 4$ and $2\le t\le \frac{m}{2}$},\\
			2t-1, &\text{if $n\ge2$ and $m=1$},\\
			2t-1, &\text{if $n \geq m\ge 2$  and $ t\ge \frac{m+1}{2}$}.
		\end{cases}
		\]
	\end{Theorem}

		\begin{Theorem}
			Let  $n \ge m \ge 1$ and $t \ge 2$ be  integers. Let $\Delta_{m,n}$ be a chessboard complex  and $I_{\Delta_{m,n}}$ its Stanley-Reisner ideal. Then 
			\[
			\mathrm{v}(I_{\Delta_{m,n}}^{(t)})=	\begin{cases}
				0, &\text{if $m=n=1$,}\\
				t +\left\lceil \frac{t-1}{n-1} \right\rceil,  &\text{if $n\ge 2$ and $m=1$},\\
				t-1+\max\left\{ d_{m,n}, \left\lceil \frac{t - 1 + c_{m,n}}{n - 1} \right\rceil  \right\}& \text{if $n\ge m\ge 2$}.
			\end{cases}
			\]
			where  $c_{m,n}=m(n-m)+\binom{m}{2}$ and 
			$d_{m,n}=
			\begin{cases}
				m-1, &\text{if $m=n$,}\\
				m, &\text{if $n>m$.}
			\end{cases}$
		\end{Theorem}

	\begin{Theorem}
		Let  $n \geq m\ge 1$ be  integers. Let $\Delta_{m,n}$ be a chessboard complex  and $\mathcal{F}(\Delta_{m,n})$ its 
		facet ideal. Then 
		\[
		\mathrm{v}(\mathcal{F}(\Delta_{m,n}))=	2m-2.
		\]
	\end{Theorem}

		\begin{Theorem}
		Let  $n \geq m\ge 1$ and  $ t\ge 2$ be integers. Let $\Delta_{m,n}$ be a chessboard complex  and $\mathcal{F}(\Delta_{m,n})$ its 
		facet ideal. Then 
		\[
	\mathrm{v}(\mathcal{F}(\Delta_{m,n})^{(t)})=	\begin{cases}
			t-1, &\text{if $m=1$,}\\
		mt,  &\text{if $n=m\ge 2$,}\\
		mt-1,&\text{if $n>m\ge 2$.}
			\end{cases}
		\]
	\end{Theorem}		
			
				\begin{Theorem}
				Let  $n \geq m\ge 1$ be integers. Let $\Delta_{m,n}$ be a chessboard complex  and $\mathcal{F}(\Delta_{m,n})$ its 
				facet ideal. Then 
				\[
				\mathrm{v}(\mathcal{F}(\Delta_{m,n})^{t})=	\begin{cases}
					t-1, &\text{if $m=1$ and $t\ge 1$,}\\
						mt-1,  &\text{if $2\le m< n\le 2m-2$ and $t\ge m$,}\\
						mt-1,  &\text{if $n\ge 2m-1$ and $t\ge 2$,}\\
					mt,&\text{if $n=m\ge 2$ and $t\ge m-1$.}
				\end{cases}
				\]
			\end{Theorem}

		The paper is organized as follows: Section \ref{sec:prelim} introduces the basic facts that are
		used throughout the paper. In Section 	\ref{sec: Stanley-Reisner ideals}, we give some exact formulas for   $\mathrm{v}$-numbers of ordinary	and symbolic powers  of the Stanley-Reisner ideal of a chessboard complex.
	In Section 	\ref{sec:facet ideals}, we  provide  some exact formulas  for $\mathrm{v}$-numbers of  symbolic powers  of  the facet ideal    of a chessboard complex. We also 	show that the $\mathrm{v}$-number function  of the ordinary powers is asymptotically linear.

		Throughout the paper, we  assume  that    $n\geq m$ are two positive integers and 
		$x_{ij}$ is the square corresponding to the $i$-th row  and  the $j$-th column in  an $m\times n$ chessboard where  $1\leq i\leq m$ and $1\leq j\leq n$. By convention,  we use the matrix $(x_{i,j})_{m\times n}$ to represent  this chessboard. For example, we can represent the  chessboard on the left  by the matrix on the right
		
		\bigskip
		\begin{center}
			\begin{tabular}{lr}
				\begin{tabular}{|p{1cm}|p{1cm}|p{1cm}|}
					\hline
					$x_{1,1}$ & $x_{1,2}$ & $x_{1,3}$\\
					& & \\
					\hline
					$x_{2,1}$ & $x_{2,2}$ & $x_{2,3}$\\
					& & \\
					\hline
					$x_{3,1}$ & $x_{3,2}$ & $x_{3,3}$\\
					& & \\
					\hline
				\end{tabular}\qquad\qquad & \qquad$\begin{pmatrix}
					x_{1,1} & x_{1,2} & x_{1,3}\\
					\\
					x_{2,1} & x_{2,2} & x_{2,3}\\
					\\
					x_{3,1} & x_{3,2} & x_{3,3}\\
				\end{pmatrix}$
			\end{tabular}
		\end{center}
		
		\vspace{1cm}

		\section{Preliminaries}
		\label{sec:prelim}
		
		In this section, we introduce the notation and recall several basic results that will be used throughout the paper.
		For two positive integers $m\le n$, let $[m,n]=\{m,m+1,\ldots,n\}$. In particular, we write $[n]$ for $[1,n]$.

		A \emph{simplicial complex} $\Delta$ on a finite set $V(\Delta)$  is a collection of subsets of $V(\Delta)$  such that if $F \in \Delta$ and $G \subseteq F$ then $G \in \Delta$.  Each element $F\in\Delta$ is called a \emph{face} of $\Delta$. A \emph{facet} of $\Delta$ is a face that is maximal with respect to inclusion. We denote the set of all facets of $\Delta$ by $\mathrm{Facets}(\Delta)$. Clearly, $\Delta$ is uniquely determined by $\Facets(\Delta)$. If $\Facets(\Delta)=\{F_1,\ldots,F_m\}$, then we write $\Delta=\langle F_1,\ldots,F_m\rangle$. A subcomplex of $\Delta$ is a simplicial complex 	whose faces are also faces of $\Delta$. If $W\subset V(\Delta)$, the induced
		subcomplex of $\Delta$ over $W$ is the subcomplex 
		\[
		\Delta[W]=\{F \in \Delta \mid F\subseteq  W\}.
		\]
		A subset $C\subseteq V(\Delta)$ is called a \emph{vertex cover} of $\Delta$ if $C\cap F\neq\emptyset$ for every $F\in\operatorname{Facets}(\Delta)$. A vertex cover $C$ is called \emph{minimal} if no proper subset of $C$ is a vertex cover of $\Delta$.
		A subset $A \subseteq V(\Delta)$ is called \textit{independent} if  $F\not\subseteq A$ for all $F \in \operatorname{Facets}(\Delta)$. An independent set is called \textit{maximal} if it is maximal with respect to inclusion. For an independent set $A$ of $\Delta$, we define the \textit{neighbor set} of $A$ in $\Delta$, denoted by $N_{\Delta}(A)$, as
		\[
		N_{\Delta}(A) = \{x \in V(\Delta) \mid \{x\} \cup A \text{ contains a facet of } \Delta\}.
		\]

		\begin{Definition}\label{twoideals}
			Let $\Delta$ be a simplicial complex with vertex set $V(\Delta) = \{x_1, \ldots, x_n\}$. We consider the polynomial ring $S = \mathbb{K}[x_1, \ldots, x_n]$ in $n$ variables over a field $\mathbb{K}$. For every subset $F \subseteq V(\Delta)$, we define the squarefree monomial $\mathbf{x}_F = \prod_{x \in F} x$. We consider the following squarefree monomial ideals associated with the simplicial complex $\Delta$:
			\begin{enumerate}
				\item The \emph{facet ideal} of $\Delta$, defined as
				\[
				\mathcal{F}(\Delta) = (\mathbf{x}_F \mid F \in \operatorname{Facets}(\Delta)) \subseteq S.
				\]
				\item The \emph{Stanley--Reisner ideal} of $\Delta$, defined as
				\[
				I_\Delta = (\mathbf{x}_F \mid F \notin \Delta) \subseteq S.
				\]
			\end{enumerate}
		\end{Definition}
		\begin{Example}\label{simplechess}
		For a chessboard complex   $\Delta_{3,3}$,  we can obtain by  Definition \ref{twoideals}  that
		\begin{enumerate}
			\item the \emph{facet ideal} of $\Delta_{3,3}$ is
			\[
			\mathcal{F}(\Delta_{3,3}) = (x_{1,1}x_{2,2}x_{3,3}, x_{1,1}x_{2,3}x_{3,2}, x_{1,2}x_{2,1}x_{3,3}, x_{1,2}x_{2,3}x_{3,1}, x_{1,3}x_{2,1}x_{3,2}, x_{1,3}x_{2,2}x_{3,1}),
			\]
			\item 	the \emph{Stanley--Reisner ideal} of $\Delta_{3,3}$ is
			\begin{align*}
				I_{\Delta_{3,3}} = (&x_{1,1}x_{1,2}, x_{1,1}x_{1,3}, x_{1,2}x_{1,3}, x_{2,1}x_{2,2}, x_{2,1}x_{2,3}, x_{2,2}x_{2,3}, x_{3,1}x_{3,2}, x_{3,1}x_{3,3}, x_{3,2}x_{3,3}, \\
				&x_{1,1}x_{2,1}, x_{1,1}x_{3,1}, x_{2,1}x_{3,1},  x_{1,2}x_{2,2}, x_{1,2}x_{3,2}, x_{2,2}x_{3,2},  x_{1,3}x_{2,3}, x_{1,3}x_{3,3}, x_{2,3}x_{3,3}).
			\end{align*}
		\end{enumerate}
	\end{Example}
	
	\medskip
	For a graded ideal $I \subseteq S$, the $t$-th symbolic power of $I$ is defined as
	\[
	I^{(t)} = \left( \bigcap_{\mathfrak{p} \in  \mathrm{Min}(I)} I^t S_{\mathfrak{p}} \right) \cap S,
	\]
	where $\mathrm{Min}(I)$ is the set of  all minimal prime ideals of $I$ and $S_{\mathfrak{p}}$ is the localization of $S$ at $\mathfrak{p}$.

	For a monomial $f = \prod_{i=1}^m \prod_{j=1}^n x_{i,j}^{s_{i,j}} \in \mathbb{K}[x_{i,j} \mid i \in [m], j \in [n]]$, we define its degree with respect to $x_{i,j}$ as $\deg_{x_{i,j}}(f) = s_{i,j}$, and its total degree as
\[
\deg(f) = \sum_{i=1}^m \sum_{j=1}^n s_{i,j}.
\]
For a prime ideal $\mathfrak{p}$ generated by a subset of the variables, we further define
\[
\deg_{\mathfrak{p}}(f) = \sum_{x_{i,j} \in \mathfrak{p}} \deg_{x_{i,j}}(f).
\]

\begin{Lemma}{\em (\cite[Proposition 2.2]{F})}\label{lowbound}
	Let $I \subset S$ be a monomial ideal. Then
	\[
	\mathrm{v}_{\mathfrak{p}}(I) \ge \alpha(I) - 1 , \text{\  \  for all \ }  \mathfrak{p} \in \Ass(I),
	\]
	where $\alpha(I) = \min \{ \deg(f) \mid f \in I \setminus \{0\} \}$.
\end{Lemma}

\begin{Lemma}{\em (\cite[Lemma 2.2]{CHJV})}\label{colon}
	Let $t$ be a positive  integer and $I \subseteq S$  a nonzero squarefree monomial ideal.  Let $J$ be a monomial ideal satisfying $I^t \subseteq J \subseteq I^{(t)}$.
	Assume that $\mathfrak{p}\in \Ass(I)$ and that $f$ is  a monomial such that
	$\deg_{\mathfrak{p}}(f) = t - 1 \quad \text{and} \quad \mathfrak{p}\subseteq(J : f)$.
	Then
	\[
	(J : f) = \mathfrak{p}.
	\]
\end{Lemma}

\begin{Lemma}{\em (\cite[Lemma 2.3]{CHJV})}\label{symboliccolon}
	Let $t$ be a positive  integer and $I \subseteq S$  a nonzero squarefree monomial ideal. Let $\mathfrak{p}\in \Ass(I)$ and  $f$  a nonzero monomial in $S$. Then
	\[
	(I^{(t)} : f) = \mathfrak{p}
	\]
	if and only if
	\[
	\deg_{\mathfrak{p}}(f) = t - 1 \quad \text{and} \quad \deg_{\mathfrak{q}}(f) \ge t
	\]
	for $\mathfrak{q}\in \Ass(I)$ such that $\mathfrak{q} \neq \mathfrak{p}$.
\end{Lemma}

	\medskip
\section{The $\mathrm{v}$-number of powers of the Stanley-Reisner ideals}
\label{sec: Stanley-Reisner ideals}

In this section, we give some exact formulas for   $\mathrm{v}$-numbers of ordinary	and symbolic powers  of the Stanley-Reisner ideal $I_{\Delta_{m,n}}$  of a chessboard complex $\Delta_{m,n}$.

\medskip

Recall that the Stanley-Reisner ideal of a chessboard complex $\Delta_{m,n}$ is
\[
I_{\Delta_{m,n}}=(x_{i,j}x_{i,k}\mid i\in[m],\,1\le j<k\le n)+(x_{p,r}x_{q,r}\mid1\le p<q\le m,r\in[n]).
\]
Therefore, if $n=m=1$, then $I_{\Delta_{1,1}} =0$, so $\mathrm{v}(I_{\Delta_{1,1}}^t)=\mathrm{v}(I_{\Delta_{1,1}}^{(t)})=0$ for all $t\ge 1$. If $n\ge 2$ and $m=1$, then $I_{\Delta_{1,n}}$ is the edge ideal of a complete graph  with $n$ vertices. 		By \cite[Theorem 5.5]{CZW},  $I_{\Delta_{1,n}}$ has a linear resolution. By \cite[Theorem 5.1]{F},   $\mathrm{v}(I_{\Delta_{1,n}}^t)=2t-1$ for all $t\ge 1$.
Moreover, by \cite[Theorem 3.3]{CHJV}, $\mathrm{v}(I_{\Delta_{1,n}}^{(t)})=t + \left\lceil \frac{t-1}{n-1} \right\rceil$ for all $t\ge 1$. 
In the following, we assume that $m\ge 2$.

For a monomial $f\in S$, we define the support of $f$ by $\supp(f)=\{x_{i,j}: x_{i,j}\mid f\}$.  For a monomial ideal $I\subset S$,  let $\mathcal{G}(I)$ be  the unique minimal set of its monomial generators,
and $\supp(I) = \bigcup_{u\in \mathcal{G}(I)}\supp(u)$.

\subsection{Ordinary  powers} In this  subsection,  we will  provide  some exact formulas for  $\mathrm{v}$-numbers of ordinary  powers  of the Stanley-Reisner ideal $I_{\Delta_{m,n}}$.  Let's first consider the first power.
\begin{Theorem}\label{vstanley}
	Let $\Delta_{m,n}$ be a chessboard complex with $n \geq m \geq 2$. Then
	\[
	\mathrm{v}(I_{\Delta_{m,n}})=
	\begin{cases}
		m-1, &\text{if $m=n$},\\
		m, &\text{if $m<n$}.
	\end{cases}
	\]
\end{Theorem}
\begin{proof}
	Let $I = I_{\Delta_{m,n}}$. Choose $\mathfrak{p} = (x_{i,j} \mid i \in [m], j \in [n], i \neq j)$, then
	\[
	\mathfrak{p} =
	\begin{cases}
		\left(I : \prod_{i=1}^{m-1} x_{i,i}\right), & \text{if } m = n, \\[1ex]
		\left(I : \prod_{i=1}^{m} x_{i,i}\right), & \text{otherwise}.
	\end{cases}
	\]
	Therefore,
	\[
	\mathrm{v}(I) \le
	\begin{cases}
		m - 1, & \text{if } m = n, \\
		m, & \text{otherwise}.
	\end{cases}
	\]
	
	Let $f$ be a monomial such that $\deg(f) = \mathrm{v}(I)$. Then there exists some $\mathfrak{q} \in \mathrm{Ass}(I)$ such that
	\[
	(I : f) = \mathfrak{q}.
	\]
	We distinguish into the following two cases:
	
	(1) Assume that $m = n$. If $\deg(f) \le m - 2$, then there exist at least two rows, say $i$ and $j$, and two columns, say $r$ and $s$, in the chessboard such that no variables from these rows and columns divide $f$. Therefore,
	\[
	\mathrm{supp}(f) \cap \{x_{i,k}, x_{j,k}, x_{k,r}, x_{k,s} \mid k \in [m]\} = \emptyset.
	\]
	By the definition of $I$, $x_{i,r}f \notin I$ and $x_{j,r}f \notin I$, implying $x_{i,r}, x_{j,r} \notin (I : f) = \mathfrak{q}$. However, again using   the definition of $I$,  $x_{i,r}x_{j,r} \in I \subseteq \mathfrak{q}$, which contradicts the fact that $\mathfrak{q}$ is a prime ideal. Thus, $\mathrm{v}(I) = \deg(f) \ge m - 1$.
	
	(2) Suppose  $m<n$.  If $\deg(f)\le m-1$, then  there are at least one row, say $\alpha$,  and two columns, say $\beta$ and $\gamma$, in the chessboard such that  no variables from these rows and columns divide $f$.
	Therefore,
	\[
	\supp(f)\cap\{x_{\alpha,\ell},x_{k,\beta},x_{k,\gamma}\mid \ell\in[n],k\in[m]\}=\emptyset.
	\]
	Thus $x_{\alpha,\beta}f, x_{\alpha,\gamma}f \notin I$, implying $x_{\alpha,\beta}, x_{\alpha,\gamma} \notin (I:f) = \mathfrak{q}$. However, $x_{\alpha,\beta}x_{\alpha,\gamma}\in I\subseteq\mathfrak{q}$, which contradicts the fact that $\mathfrak{q}$ is prime.  Thus, $\mathrm{v}(I)=\deg(f)\ge m$.
\end{proof}

		We first  consider the $\mathrm{v}$-numbers of ordinary powers  of  the Stanley-Reisner ideal of a chessboard complex $\Delta_{m,n}$.
		Let $t\ge 2$ be an integer.
		
		\begin{Lemma}\label{lowbound1}
			Let $\Delta_{m,n}$ be a chessboard complex with $n \geq m \geq 2$. Then
			\[
			\mathrm{v}(I_{\Delta_{m,n}}^t)\ge \max\{2t-1,m\} .
			\]
		\end{Lemma}
		\begin{proof}
			Let $I = I_{\Delta_{m,n}}$. First, we  show that $\mathrm{v}(I^t)\ge m$. Suppose for a contradiction that $\mathrm{v}(I^t)\le m-1$. By the definition of the  $\mathrm{v}$-number, there exists  some monomial $f\notin I^t$ of degree $d \le m-1$ such that $\mathfrak{p} = (I^t : f)$, where $\mathfrak{p} \in \Ass(I^t)$.
			We can write
			\begin{align*}
				f =\prod\limits_{k=1}^{d} x_{i_k,j_k},
				\tag{1}\label{eq:my_equation1}
			\end{align*}
			where $i_k \in [m]$ and $j_k \in [n]$ for all $1 \le k \le d$.
			We claim that $f\in I^{t-1}$.
			
			Indeed, since $\mathfrak{p} = (I^t : f)$,  $x_{a,b}f\in I^t$ for any $x_{a,b}\in \mathfrak{p}$. Thus we   can write
			\[
			x_{a,b}f=gf_1\cdots f_t,
			\]
			where $f_1,\ldots,f_t\in \mathcal{G}(I)$  and  $g$ is a monomial. If $x_{a,b}\mid f_k$ for some $k\in[t]$, then 
			\[
			f = g \left(\frac{f_k}{x_{a,b}}\right) \prod_{j \neq k} f_j \in I^{t-1}.
			\]
			Otherwise, $x_{a,b}\mid g$ and $f=(g/x_{a,b})f_1\cdots f_t\in I^t$.
			
			By the definition of $I$,  we know that
			in the expression (\ref{eq:my_equation1}) of $f$, there exist $1\le r<s\le d$
			such that  $i_r=i_s$ or $j_r=j_s$.  Without loss of generality, let us assume that $i_r=i_s$.
			Thus $|[m]\setminus\{i_1,\ldots,i_d\}|\ge 2$ and $|[n]\setminus\{j_1,\ldots,j_d\}|\ge 1$. Therefore, there exist distinct $\alpha_1,\alpha_2\in [m]\setminus\{i_1,\ldots,i_d\}$ and $\beta\in [n]\setminus\{j_1,\ldots,j_d\}$. This forces $x_{\alpha_1,\beta}x_{\alpha_2,\beta}\in I$, so $x_{\alpha_1,\beta}x_{\alpha_2,\beta}f\in I^t$. This implies that  $x_{\alpha_1,\beta}x_{\alpha_2,\beta}\in(I^t:f)=\mathfrak{p}$. Thus, $x_{\alpha_1,\beta}\in\mathfrak{p}$ or $x_{\alpha_2,\beta}\in\mathfrak{p}$.
			
			If $x_{\alpha_1,\beta}\in\mathfrak{p}$, then $x_{\alpha_1,\beta}f\in I^t$. Thus we can write
			\begin{align*}
				x_{\alpha_1,\beta}f=hu_1\cdots u_t,
				\tag{2}\label{eq:my_equation2}
			\end{align*}
			where $u_1,\ldots,u_t\in \mathcal{G}(I)$ and  $h$ is a monomial.  Since $\alpha_1\notin\{i_1,\ldots,i_d\}$ and $\beta\notin \{j_1,\ldots,j_d\}$,  $x_{\alpha_1,\beta}\nmid u_k$ for all  $k\in[t]$. It follows from the expression (\ref{eq:my_equation2}) of $x_{\alpha_1,\beta}f$
			that $x_{\alpha_1,\beta}\mid h$. Therefore, $f=(h/x_{\alpha_1,\beta})u_1\cdots u_t\in I^t$, which contradicts the assumption that $f\notin I^t$. Thus, $x_{\alpha_1,\beta}\notin\mathfrak{p}$. By the same arguments, we also can obtain  $x_{\alpha_2,\beta}\notin\mathfrak{p}$. This contradicts that $x_{\alpha_1,\beta}\in\mathfrak{p}$ or $x_{\alpha_2,\beta}\in\mathfrak{p}$.
			
			Therefore, $\mathrm{v}(I^t)\ge m$. Combining this with Lemma \ref{lowbound}, we conclude that $\mathrm{v}(I^t)\ge \max\{2t-1,m\}$.
		\end{proof}

		\begin{Theorem}
			Let $n \geq m \geq 4$ and $t$ be three integers such that  $2\le t\le \frac{m}{2}$.  Suppose  that
			$\Delta_{m,n}$ is a chessboard complex. Then
			\[
			\mathrm{v}(I_{\Delta_{m,n}}^t)=m .
			\]
		\end{Theorem}
			\begin{proof}
			Let $I = I_{\Delta_{m,n}}$. Then $\operatorname{v}(I^t) \ge m$ by Lemma 3.2. To prove the equality holds, it suffices to show that $\operatorname{v}(I^t) \le m$.

			Choose $f = (\prod\limits_{i=1}^{2t-1} x_{i,1})(\prod\limits_{j=2t}^{m} x_{j, j-2t+2})$, then, for  any $\ell\in [2t-1]$, we have
			\[
			x_{\ell,1}f=
			\begin{cases}
				(x_{\ell,1}x_{2t-1,1})(\prod\limits_{i=1}^{t-1} x_{2i-1,1}x_{2i,1})(\prod\limits_{j=2t}^{m} x_{j, j-2t+2}),  &\text{if $\ell=1$,}\\
				(x_{1,1}x_{\ell,1})(\prod\limits_{i=1}^{t-1} x_{2i,1}x_{2i+1,1})(\prod\limits_{j=2t}^{m} x_{j, j-2t+2}), &\text{if $\ell\in[2,2t-1]$.}
			\end{cases}
			\]
			Therefore, $x_{\ell,1}f\in I^t$, which implies $x_{\ell,1}\in (I^t:f)$. For any $r\in [2t-1]$ and $s\in[2,n]$, we have
			\[
			x_{r,s}f=
			\begin{cases}
				u(\prod\limits_{i=1}^{\frac{r-1}{2}}x_{2i-1,1}x_{2i,1})(\prod\limits_{j=\frac{r+1}{2}}^{t-1} x_{2j,1}x_{2j+1,1}),  &\text{if $r$ is odd,}\\
				u(x_{r-1,1}x_{r+1,1})(\prod\limits_{i=1}^{\frac{r-2}{2}}x_{2i-1,1}x_{2i,1})(\prod\limits_{j=\frac{r+2}{2}}^{t-1} x_{2j,1}x_{2j+1,1}), &\text{if $r$ is even,}
			\end{cases}
			\]
			where $u=(x_{r,1}x_{r,s})(\prod\limits_{i=2t}^m x_{i, i-2t+2})$.
			Thus, $x_{r,s}f\in I^t$. It follows that $x_{r,s}\in (I^t:f)$. For any $\alpha\in [2t,m]$ and $\beta\in[n]\setminus\{\alpha-2t+2\}$, we  also have $x_{\alpha,\beta}\in (I^t:f)$, since
			\[
			x_{\alpha,\beta}f=x_{2t-1,1}(\prod\limits_{i\in[2t,m]\setminus\{\alpha\}} x_{i, i-2t+2})(x_{\alpha,\beta}x_{\alpha,\alpha-2t+2})(\prod\limits_{j=1}^{t-1} x_{2j-1,1}x_{2j,1})\in I^t.
			\]
			Thus, $\mathfrak{p}\subseteq (I^t : f)$, where  $\mathfrak{p} = (x_{\alpha,\beta} \mid x_{\alpha,\beta}\in V(\Delta_{m,n})\setminus\{x_{k, k-2t+2}\mid k\in[2t,m]\})$.

			Next, we will prove that the inverse inclusion relation holds.
			It suffices to show that $gf\notin I^t$ for any   $g\notin \mathfrak{p}$.  Suppose for a contradiction that there exists some monomial $g\notin \mathfrak{p}$ such that $gf\in I^t$.  We can write $g=\prod\limits_{i=2t}^m x_{i, i-2t+2}^{s_i}$,  thus
			\begin{align*}
			gf = \left(\prod_{i=2t}^m x_{i,i-2t+2}^{s_i+1}\right) \left(\prod_{j=1}^{2t-1} x_{j,1}\right) = h u_1 \cdots u_t,
				\tag{3}\label{eq:my_equation3}
			\end{align*}
			where $u_1,\ldots,u_t\in \mathcal{G}(I)$ and $h$ is a monomial.
			
			From the structure of the chessboard, we can see that, for each $i\in[2t,m]$,
			$\supp(gf)\cap M_i=\emptyset$, where $M_i=\{x_{i,\ell}\mid \ell\in[n]\setminus\{i-2t+2\}\}\cup\{x_{\ell,i-2t+2}\mid \ell\in[m]\setminus\{i\}\}$.
			From the expression (\ref{eq:my_equation3}) of $gf$, we know that
			$x_{i, i-2t+2}\nmid u_k$ for all $k\in[t]$,  implying  $x_{i,i-2t+2}^{s_i+1} \mid h$.
 Consequently, $\prod\limits_{j=1}^{2t-1} x_{j,1}=\left(h/(\prod\limits_{i=2t}^m x_{i, i-2t+2}^{s_i+1})\right)u_1\cdots u_t\in I^t$, which is a contradiction.
			
	Therefore, $(I^t : f)=\mathfrak{p}$, which means that $\operatorname{v}(I^t) \le \operatorname{deg}(f) = m$.
			\end{proof}
	\begin{Theorem}
	Let $ t\ge \frac{m+1}{2}$ be an integer and  $\Delta_{m,n}$   a chessboard complex. Then
	\[
	\mathrm{v}(I_{\Delta_{m,n}}^t)=2t-1 .
	\]
\end{Theorem}
\begin{proof}	
			Let $I = I_{\Delta_{m,n}}$. Then  $\mathrm{v}(I^t)\ge 2t-1$  by Lemma \ref{lowbound1}. We will prove $\mathrm{v}(I^t)\le 2t-1$.
		We distinguish into the following  three cases:
		
		(1) If $m=2$, then by Theorem \ref{vstanley}, $\mathfrak{p}=(x_{i,j}\mid i\in[m],j\in[n],i\neq j) \in \Ass(I)$. 
		We choose  $f_1=x_{2,2}(x_{1,1}x_{2,1})^{t-1}$. Then, from the choice of $f_1$ and $\mathfrak{p}$,  $\supp(\mathfrak{p})\cap\supp(f_1)=\{x_{2,1}\}$. Therefore, 
		$\deg_\mathfrak{p}(f_1)=\deg_{x_{2,1}}(f_1)=t-1$. On the other hand, for any $\ell\in[2,n]$, $x_{1,\ell}f_1=(x_{1,1}x_{1,\ell})(x_{2,1}x_{2,2})(x_{1,1}x_{2,1})^{t-2}\in I^t$,   so $x_{1,\ell}\in(I^t:f_1)$. At the same time,  for any $q\in[n]\setminus\{2\}$, $x_{2,q}f_1=(x_{2,2}x_{2,q})(x_{1,1}x_{2,1})^{t-1}\in I^t$, hence
		$x_{2,q}\in(I^t:f_1)$. Therefore, $\mathfrak{p}\subseteq (I^t:f_1)$. By Lemma \ref{colon},   $\mathfrak{p}= (I^t:f_1)$.
		By the definition of $\mathrm{v}$-number, $\mathrm{v}(I^t)\le \operatorname{deg}(f_1) = 2t-1$.
		
			(2) If  $m\ge 4$ and it is even,  then we choose  $f_2=x_{m,1}^2(x_{1,1}x_{2,1})^{t-\frac{m}{2}-1}(\prod\limits_{i=1}^{m-1}x_{i,1})$. Thus, for any $\ell\in [m]$, we have
		\[
	x_{\ell,1}f_2=
	\begin{cases}
		(x_{1,1}x_{2,1})^{t-\frac{m}{2}-1}(x_{\ell,1}x_{m,1})(\prod\limits_{i=1}^{\frac{m}{2}} x_{2i-1,1}x_{2i,1}),  &\text{if $\ell\in[m-1]$,}\\
		(x_{1,1}x_{2,1})^{t-\frac{m}{2}-1}(\prod\limits_{i=1}^{\frac{m}{2}-2} x_{2i-1,1}x_{2i,1})\left(\prod\limits_{j=m-3}^{m-1} (x_{j,1}x_{m,1})\right),  &\text{if $\ell=m$.}
	\end{cases}
	\]
		Thus, $x_{\ell,1}f_2\in I^t$, which implies $x_{\ell,1}\in (I^t:f_2)$.
		Furthermore, for any $r\in [m]$ and $s\in[2,n]$, we have
	\[
	x_{r,s}f_2=
	\begin{cases}w(x_{r-1,1}x_{m,1})(\prod\limits_{i\in[\frac{m}{2}]\setminus\{\frac{r}{2}\}} x_{2i-1,1}x_{2i,1}), &\text{if $r$ is even,}\\
		w(x_{r+1,1}x_{m,1})(\prod\limits_{i\in[\frac{m}{2}]\setminus\{\frac{r+1}{2}\}} x_{2i-1,1}x_{2i,1}),  &\text{if $r< m-1$ is odd,}\\
		w(\prod\limits_{i=2}^{\frac{m}{2}-1} x_{2i-1,1}x_{2i,1})\left(\prod\limits_{j=1}^{2}(x_{j,1}x_{m,1})\right),  &\text{if $r=m-1$.}
	\end{cases}
	\]
		where $w=(x_{r,1}x_{r,s})(x_{1,1}x_{2,1})^{t-\frac{m}{2}-1}$.  Thus $x_{r,s}f_2\in I^t$,  i.e.,  $x_{r,s}\in (I^t:f_2)$.
		Therefore, $\mathfrak{m}\subseteq (I^t:f_2)$, where $\mathfrak{m}=(x_{i,j}\mid i\in[m],j\in[n])$ is the unique maximal homogeneous ideal of $S$.  Thus  $(I^t:f_2)=\mathfrak{m}$, since it is trivial that $(I^t:f_2)\subseteq\mathfrak{m}$. By the definition of $\mathrm{v}$-number, $\mathrm{v}(I^t)\le \operatorname{deg}(f_2) = 2t-1$.
		
		(3) If $m$ is odd, then we choose  $f_3=(x_{1,1}x_{2,1})^{t-\frac{m+1}{2}}(\prod\limits_{i=1}^{m}x_{i,1})$. Thus, for any $\gamma\in [m]$, we have
	\[
	x_{\gamma,1}f_3=
	\begin{cases}
		(x_{1,1}x_{2,1})^{t-\frac{m+1}{2}}(x_{\gamma,1}x_{m,1})(\prod\limits_{i=1}^{\frac{m-1}{2}} x_{2i-1,1}x_{2i,1}),  &\text{if $\gamma\in[m-1]$,}\\
		(x_{1,1}x_{2,1})^{t-\frac{m+1}{2}}(\prod\limits_{i=1}^{\frac{m-3}{2}} x_{2i-1,1}x_{2i,1})\left(\prod\limits_{j=m-2}^{m-1} (x_{j,1}x_{m,1})\right),  &\text{if $\gamma=m$.}
	\end{cases}
	\]
		Hence, $x_{\gamma,1}f_3\in I^t$, which implies $x_{\gamma,1}\in (I^t:f_3)$. At the same time,  for any $\alpha\in [m]$,  $\beta\in[2,n]$, we have
		\[
		x_{\alpha,\beta}f_3=
		\begin{cases}
			z(\prod\limits_{i=1}^{\frac{\alpha-1}{2}}x_{2i-1,1}x_{2i,1})(\prod\limits_{j=\frac{\alpha+1}{2}}^{\frac{m-1}{2}} x_{2j,1}x_{2j+1,1}),  &\text{if $\alpha$ is odd,}\\
			z(x_{\alpha-1,1}x_{\alpha+1,1})(\prod\limits_{i=1}^{\frac{\alpha}{2}-1}x_{2i-1,1}x_{2i,1})(\prod\limits_{j=\frac{\alpha}{2}+1}^{\frac{m-1}{2}} x_{2j,1}x_{2j+1,1}), &\text{if $\alpha$ is even,}
		\end{cases}
		\]
		where $z=(x_{\alpha,1}x_{\alpha,\beta})(x_{1,1}x_{2,1})^{t-\frac{m+1}{2}}$. Thus, $x_{\alpha,\beta}f_3\in I^t$, implying  $x_{\alpha,\beta}\in (I^t:f_3)$.
		Therefore, $\mathfrak{m}\subseteq (I^t:f_3)$.   Thus $(I^t:f_3)=\mathfrak{m}$, since it is clear that $(I^t:f_3)\subseteq\mathfrak{m}$.
		Again by the definition of $\mathrm{v}$-number, $\mathrm{v}(I^t)\le \operatorname{deg}(f_3) = 2t-1$.
			\end{proof}

	\subsection{Symbolic powers}
	
	In the following, we will consider the $\mathrm{v}$-numbers of symbolic powers  of  the Stanley-Reisner ideal $I_{\Delta_{m,n}}$.  
	Let $t\ge 2$ be an integer.
	
	By \cite[Theorem 1.7]{MS},  we know that 
	\[
	I_{\Delta_{m,n}}=\bigcap\limits_{F \in \mathrm{Facets}(\Delta_{m,n})} \mathfrak{q}_F,
	\]
	where $\mathfrak{q}_F=(x_{i,j}\mid x_{i,j}\in V(\Delta_{m,n})\setminus F)$. Thus the $t$-th symbolic power of
	$I_{\Delta_{m,n}}$ is 
	\[
	I_{\Delta_{m,n}}^{(t)}=\bigcap\limits_{F \in \mathrm{Facets}(\Delta_{m,n})} \mathfrak{q}_F^t. 
	\]
	
\begin{Lemma}\label{symboliclowbound1}
	Let $t\ge 2$ be an integer and $\Delta_{m,n}$ a chessboard complex with $n \geq m \geq 2$. 
	Then 
	\[
	\mathrm{v}(I_{\Delta_{m,n}}^{(t)})\ge t-1+
	\max\left\{d_{m,n},
	\left\lceil\frac{t-1+c_{m,n}}{n-1}\right\rceil\right\}, 
	\]
	where  $c_{m,n}=m(n-m)+\binom{m}{2}$ and 
	$d_{m,n}=
	\begin{cases}
		m-1, &\text{if $m=n$,}\\
		m, &\text{if $n>m$.}
	\end{cases}$
\end{Lemma}
	\begin{proof}
		Let $I=I_{\Delta_{m,n}}$ and $\Gamma_{m,n,t}=t-1+	\max\left\{d_{m,n},	\left\lceil\frac{t-1+c_{m,n}}{n-1}\right\rceil\right\}$.
	By the definition of $\mathrm{v}$-number, 
	\[
	\mathrm{v}(I^{(t)})=\min\{\mathrm{v}_{\mathfrak{q}_{F}}(I^{(t)})\mid F\in \Facets(\Delta_{m,n})\}. 
	\]
	We will show that for any $F\in\mathrm{Facets}(\Delta_{m,n})$,
	$\mathrm{v}_{\mathfrak{q}_F}(I^{(t)})\ge \Gamma_{m,n,t}$.

		Let $F=\{x_{1,q_1},\ldots,x_{m,q_m}\}$.
	By the definition of  the local $\mathrm{v}$-number, there exists a monomial $f=\prod\limits_{i=1}^{m}\prod\limits_{j=1}^{n}x_{i,j}^{s_{i,j}}$ of degree  $\mathrm{v}_{\mathfrak{q}_{F}}(I^{(t)})$ such that  $(I^{(t)}:f)=\mathfrak{q}_{F}$. 
	By Lemma \ref{symboliccolon},   $\sum\limits_{x_{i,j}\notin F}s_{i,j}=t-1$. Thus
	\[
	\deg(f)=\sum\limits_{x_{i,j}\notin F}s_{i,j}+\sum\limits_{i\in[m]}s_{i,q_i}=t-1+\sum\limits_{i\in[m]}s_{i,q_i}.  
	\]
	To prove $\deg(f)\ge \Gamma_{m,n,t}$, it suffices to show that  $\sum\limits_{i\in[m]}s_{i,q_i}\ge \max\left\{d_{m,n},\left\lceil\frac{t-1+c_{m,n}}{n-1}\right\rceil\right\}$. 
	Now, we  prove that $\sum\limits_{i\in[m]}s_{i,q_i}\ge\left\lceil\frac{t-1+c_{m,n}}{n-1}\right\rceil$.  
	It is enough to prove that 
	\[
	t-1\le (n-1)\sum\limits_{i\in[m]}s_{i,q_i}-c_{m,n}.
	\]
	Note that  $\sum\limits_{x_{i,j}\notin F}s_{i,j}=t-1$ and 
	$\sum\limits_{x_{i,j}\notin F}s_{i,j}=\sum\limits_{1\le i<j\le m}(s_{i,q_j}+s_{j,q_i})+\sum\limits_{i=1}^{m}\sum\limits_{q\in[n]\setminus\{q_1,\ldots,q_m\}}s_{i,q}$.
	It suffices to show
	\begin{align*}
		\sum\limits_{1\le i<j\le m}(s_{i,q_j}+s_{j,q_i})+\sum\limits_{i=1}^{m}\sum\limits_{q\in[n]\setminus\{q_1,\ldots,q_m\}}s_{i,q}\le (n-1)\sum\limits_{i\in[m]}s_{i,q_i}-c_{m,n}.
		\tag{4}\label{eq:my_equation4}
	\end{align*}
	To prove that (\ref{eq:my_equation4}) holds, we only need to show that 
	$\sum\limits_{1\le i<j\le m}(s_{i,q_j}+s_{j,q_i})\le (m-1)(\sum\limits_{i\in[m]}s_{i,q_i})-\binom{m}{2}$ and 
	$\sum\limits_{k=1}^{m}\sum\limits_{q\in[n]\setminus\{q_1,\ldots,q_m\}}s_{k,q}\le (n-m)(\sum\limits_{i\in[m]}s_{i,q_i}-m)$.
	
	\medskip
	First, we prove  $\sum\limits_{1\le i<j\le m}(s_{i,q_j}+s_{j,q_i})\le (m-1)(\sum\limits_{i\in[m]}s_{i,q_i})-\binom{m}{2}$. 
	
	For any $1\le i<j\le m$, we consider a facet  $F_{ij}=(F\setminus\{x_{i,q_i},x_{j,q_j}\})\cup\{x_{i,q_j},x_{j,q_i}\}$ of  $\Delta_{m,n}$ that is distinct from $F$.  Applying Lemma \ref{symboliccolon}  once again, we have $\sum\limits_{x_{u,v}\notin F_{ij}}s_{u,v} \ge t$.
	Thus
	\begin{align*}
		\sum\limits_{x_{u,v}\notin F_{ij}}s_{u,v} &= \sum\limits_{x_{u,v}\notin F}s_{u,v} + s_{i,q_i} + s_{j,q_j} - s_{i,q_j} - s_{j,q_i}\\
		&=(t-1) + s_{i,q_i} + s_{j,q_j} - s_{i,q_j} - s_{j,q_i}\ge t.
	\end{align*}
	Therefore,  $s_{i,q_j}+s_{j,q_i}\le s_{i,q_i}+s_{j,q_j}-1$.
	It follows that
	\[
	\sum\limits_{1\le i<j\le m}(s_{i,q_j}+s_{j,q_i})\le \sum\limits_{1\le i<j\le m}(s_{i,q_i}+s_{j,q_j}-1)=(m-1)(\sum\limits_{i\in[m]}s_{i,q_i})-\binom{m}{2}.
	\] 
	\medskip
	Next, we will show that  $\sum\limits_{k=1}^{m}\sum\limits_{\ell\in[n]\setminus\{q_1,\ldots,q_m\}}s_{k,\ell}\le (n-m)(\sum\limits_{i\in[m]}s_{i,q_i}-m)$.
	
	The case $n=m$ is trivial. Suppose  $n>m$.
	For any $k\in[m]$ and any $\ell\in[n]\setminus\{q_1,\ldots,q_m\}$, we consider a facet  $F_{k\ell}'=(F\setminus\{x_{k,q_k}\})\cup\{x_{k,\ell}\}$ of  $\Delta_{m,n}$ that is distinct from $F$.  Since  $\sum\limits_{x_{i,j}\notin F}s_{i,j}=t-1$ and $\sum\limits_{x_{u,v}\notin F_{k\ell}'}s_{u,v} \ge t$,
	\begin{align*}
		\sum\limits_{x_{u,v}\notin F_{k\ell}'}s_{u,v} = \sum\limits_{x_{u,v}\notin F}s_{u,v} + s_{k,q_k} - s_{k,\ell}=(t-1) + s_{k,q_k} - s_{k,\ell}\ge t.
	\end{align*}
	Thus  $s_{k,\ell}\le s_{k,q_k}-1$.
	Therefore, 
	\[
	\sum\limits_{k=1}^{m}\sum\limits_{\ell\in[n]\setminus\{q_1,\ldots,q_m\}}s_{k,\ell}\le \sum\limits_{k=1}^{m}\sum\limits_{\ell\in[n]\setminus\{q_1,\ldots,q_m\}}(s_{k,q_k}-1) =(n-m)(\sum\limits_{i\in[m]}s_{i,q_i}-m).
	\] 
	\bigskip
	On the other hand, we will prove  that $\sum\limits_{i\in[m]}s_{i,q_i}\ge d_{m,n}$.
	From the proof above, we can see that:  if $1\le i<j\le m$,  then $s_{i,q_i}+s_{j,q_j}\ge 1$, and if $n>m$, then 
	$s_{k,q_k}\ge 1$ for any  $k\in[m]$ and $\ell\in[n]\setminus\{q_1,\ldots,q_m\}$.
	Thus,  if $1\le i<j\le m$,  then 
	there is at most one $i\in[m]$ such that $s_{i,q_i}=0$. Therefore,  $\sum\limits_{i\in[m]}s_{i,q_i}\ge m-1$.  If 
	$n>m$, then $\sum\limits_{i\in[m]}s_{i,q_i}\ge m$. Therefore,  $\sum\limits_{i\in[m]}s_{i,q_i}\ge d_{m,n}$. 
\end{proof}
\begin{Theorem}\label{symboliclpower}
	Let $t\ge 2$ be an integer and $\Delta_{m,n}$ a chessboard complex with $n \geq m \geq 2$. 
	Then 
	\[
	\mathrm{v}(I_{\Delta_{m,n}}^{(t)})= t-1+
	\max\left\{d_{m,n},
	\left\lceil\frac{t-1+c_{m,n}}{n-1}\right\rceil\right\}, 
	\]
	where  $c_{m,n}$ and  $d_{m,n}$ are  defined as  in Lemma \ref{symboliclowbound1}.
\end{Theorem}
	\begin{proof}
	Let $I=I_{\Delta_{m,n}}$ and  $F_0=\{x_{1,1}, \ldots ,x_{m,m}\}$. Then $\mathfrak{q}_{F_0}=(x_{i,j}\mid x_{i,j}\in V(\Delta_{m,n})\setminus F_0)$. 
	For any $i\in[m]$ and any $j\in[n]$, we define 
	\[	
	\alpha_i=
	\begin{cases}
		\rho_{m,n,t}-m+1,&\text{if $i=1$,}\\
		1,    &\text{if $2\le i\le m$,}
	\end{cases}
	\]
	and 
	\[
	\beta_{i,j}=
	\begin{cases}
		\alpha_i,  &\text{if $1\le i\le j\le m$,}\\
		\alpha_i-1,&\text{if $1\le j<i$  or $m<j\le n$,} 
	\end{cases}
	\]
	where $\rho_{m,n,t}=\max\left\{d_{m,n},\left\lceil\frac{t-1+c_{m,n}}{n-1}\right\rceil\right\}$. 
	Thus 
	\[
	\alpha_1=\rho_{m,n,t}-m+1\ge d_{m,n}-m+1=\begin{cases}
		0,  &\text{if $n=m$,}\\
		1,&\text{if $n>m$,} 
	\end{cases}
	\]
	and
	\begin{align*}
		\sum_{x_{i,j}\notin F_0}\beta_{i,j}
		&=\sum\limits_{i=1}^{m}\big[(n-m)(\alpha_i-1)+(i-1)(\alpha_i-1)+(m-i)\alpha_i\big]\\
		&=(n-1)\rho_{m,n,t}-c_{m,n}\ge t-1.
	\end{align*}
	Choose a  monomial   $f=\prod_{i=1}^{m}\prod_{j=1}^{n}x_{i,j}^{s_{i,j}}$, where $s_{i,j}\ge 0$   such that $s_{i,i}=\alpha_i$, $s_{i,j}\le \beta_{i,j}$ for any $j\ne i$ and $\sum_{x_{i,j}\notin F_0}s_{i,j}=t-1$.  Then $\deg(f)=\sum\limits_{i=1}^{m}\sum\limits_{j=1}^{n}s_{i,j}$.
	We would like to  prove that $(I^{(t)}:f)=\mathfrak{q}_{F_0}$. 
	By Lemma \ref{symboliccolon},  this is equivalent to verifying
	that  	
	\[
	\deg_{\mathfrak{q}_{F_0}}(f) = t - 1 \quad \text{and} \quad \deg_{\mathfrak{q}_F}(f) \ge t
	\]
	for any $\mathfrak{q}_F\in \Ass(I)$ such that $\mathfrak{q}_F \neq \mathfrak{q}_{F_0}$.
	
	It is clear that $\deg_{\mathfrak{q}_{F_0}}(f) =\sum_{x_{i,j}\notin F_0}s_{i,j}=t-1$. For any $\mathfrak{q}_F\in \Ass(I)$ such that $\mathfrak{q}_F \neq \mathfrak{q}_{F_0}$, we will verify $\deg_{\mathfrak{q}_F}(f) \ge t$.
	Let  $F=\{x_{1,\ell_1},\ldots,x_{m,\ell_m}\}$, then 
	\begin{align*}
		\deg_{\mathfrak{q}_F}(f)&=\sum\limits_{x_{i,j}\notin F}s_{i,j}=\deg (f)-\sum_{i=1}^{m}s_{i,\ell_i}=\sum\limits_{i=1}^{m}\sum\limits_{j=1}^{n}s_{i,j}
		-\sum_{i=1}^{m}s_{i,\ell_i}\\
		&=
		\sum\limits_{x_{i,j}\notin F_0}s_{i,j}+\sum_{i=1}^{m}s_{i,i}-\sum_{i=1}^{m}s_{i,\ell_i}=(t-1)+\sum_{i=1}^{m}\alpha_i-\sum_{i=1}^{m}s_{i,\ell_i}.
	\end{align*}
	Therefore,  $\deg_{\mathfrak{q}_F}(f) \ge t$ if and only if $\sum_{i=1}^{m}s_{i,\ell_i}\le\sum_{i=1}^{m}\alpha_i-1$.
	
	Therefore, by the definition of $\mathrm{v}$-number, $\mathrm{v}(I^{(t)})\le \operatorname{deg}(f) =  t-1+\rho_{m,n,t}$.  By Lemma \ref{symboliclowbound1}, we have $\mathrm{v}(I^{(t)})=\operatorname{deg}(f) =  t-1+\rho_{m,n,t}$.
\end{proof}

\section{The $\mathrm{v}$-number of powers of the facet ideals}
\label{sec:facet ideals}

In this section, we  provide  some exact formulas  for $\mathrm{v}$-numbers of  symbolic powers  of  the facet ideal  $\mathcal{F}(\Delta_{m,n})$  of a chessboard complex $\Delta_{m,n}$. We also 
show that the $\mathrm{v}$-number function of the ordinary powers is asymptotically linear.

Recall that the facet ideal of $\Delta_{m,n}$ is
\[
\mathcal{F}(\Delta_{m,n}) = (x_{1,\ell_1} \cdots x_{m,\ell_m} \mid 1 \le \ell_1, \dots, \ell_m \le n \text{ are pairwise distinct}).
\]
If $m = 1$, then $\mathcal{F}(\Delta_{1,n}) = (x_{1,1}, \dots, x_{1,n})$.  By \cite[Corollary 4.4 and Proposition 4.13]{BM}, $\mathrm{v}(\mathcal{F}(\Delta_{1,n})^{(t)}) =\mathrm{v}(\mathcal{F}(\Delta_{1,n})^t) = t - 1$ for all $t \ge 1$. From now on, we assume that $m \ge 2$ and adopt the convention that $\displaystyle \prod_{i=k}^{\ell} \prod_{j=\alpha}^{\beta} x_{i,j} = 1$ if $k > \ell$ or $\alpha > \beta$.  In order to obtain a lower bound on the $\mathrm{v}$-number of  the facet ideal of $\Delta_{m,n}$,  we need to use the two sets and a lemma introduced in \cite{JV}.

\[
\mathcal{M}_{\Delta}=\{A \mid A \text{ is a maximal independent set of } \Delta\}, \text{ and}
\]
\[
\mathcal{A}_{\Delta}=\{A \mid A \text{ is an independent set of } \Delta \text{ and } N_{\Delta}(A) \text{ is a minimal vertex cover of } \Delta\}.
\]

\medskip
In the following, we will consider the $\mathrm{v}$-numbers of symbolic powers  of  the facet ideal  $\mathcal{F}(\Delta_{m,n})$.  
We need the following lemma.

\begin{Lemma}{\em (\cite[Theorem 3.5]{JV})}\label{edgeideal}
	Let $I$ be the facet ideal of a simplicial complex $\Delta$. If $I$ is not prime, then $\mathcal{M}_{\Delta}\subseteq \mathcal{A}_{\Delta}$ and
	\[
	\mathrm{v}(I) = \min\{|A| : A \in \mathcal{A}_{\Delta}\}.
	\]
\end{Lemma}
	
	\medskip
	\subsection{Symbolic powers}
		
		\begin{Lemma}\label{lowbound2}
			Let $\Delta_{m,n}$ be a chessboard complex with $n \geq m \geq 2$. Then
			\[
			\mathrm{v}(\mathcal{F}(\Delta_{m,n}))\ge 2m-2.
			\]
		\end{Lemma}
		\begin{proof}
			For any  $A\in \mathcal{A}_{\Delta_{m,n}}$,  by \cite[Lemma 3.4(a)]{JV},   $(\mathcal{F}(\Delta_{m,n}):\prod\limits_{x_{i,j}\in A}x_{i,j})=(x_{i,j}\mid x_{i,j}\in N_{\Delta_{m,n}}(A))$.  Thus $(x_{i,j}\mid x_{i,j}\in N_{\Delta_{m,n}}(A))\in  \Ass(\mathcal{F}(\Delta_{m,n}))$. By the definition of $\mathcal{A}_{\Delta_{m,n}}$,
			$N_{\Delta_{m,n}}(A)$   is a minimal vertex cover of  $\Delta_{m,n}$.
			By \cite[Theorem 3.3]{JZWZ}, $N_{\Delta_{m,n}}(A)$ can be  obtained from $V(\Delta_{m,n})$ by deleting all elements in $s$ rows, say
			$i_1, \dots, i_s$ rows,   and all elements in $m-1-s$ columns, say $j_1, \dots, j_{m-1-s}$ columns,
			where $0\le s\le m-1$. By convention, we set $\{i_1, \dots, i_s\} = \emptyset$ if $s = 0$, and $\{j_1, \dots, j_{m-1-s}\} = \emptyset$ if $s = m - 1$. 
			
			We consider the following  two  sets of points
			\begin{align*}
				S_1&=\big\{x_{i,j}\mid i\in[m]\setminus\{i_1,\ldots,i_s\}\text{\ and\ } j\in \{j_1,\ldots,j_{m-1-s}\}\big\},\\
				S_2&=\big\{x_{i,j}\mid i\in\{i_1,\ldots,i_s\}\text{\ and\ } j\in [n]\setminus\{j_1,\ldots,j_{m-1-s}\}\big\}
			\end{align*}
			on the chessboard. Let $A_1=A\cap S_1$ and  $A_2=A\cap S_2$, then $|A|\ge |A_1|+|A_2|$.
			
			\medskip
			To prove that $|A|\ge 2m-2$, it suffices  to show that $|A_{1}|\ge 2(m-1-s)$ and $|A_{2}|\ge 2s$.
			
			First, we
			show that $|A_{1}|\ge 2(m-1-s)$.
			The case $s=m-1$ is trivial. Suppose  $s<m-1$.
			By the definition of $N_{\Delta_{m,n}}(A)$,  for any vertex $x_{a,b}\in N_{\Delta_{m,n}}(A)$,   $\{x_{a,b}\}\cup A$  contains a facet of $\Delta_{m,n}$, say $ F_{a,b}$.  Furthermore,  $x_{a,b}\in F_{a,b}$, since  $A$  is an independent set of $\Delta_{m,n}$.
			Thus, $	F_{a,b}\setminus\{x_{a,b}\}\subseteq A$. From the position of $ N_{\Delta_{m,n}}(A)$ on the chessboard, we can see that
			$V(\Delta_{m,n})\setminus N_{\Delta_{m,n}}(A)=\{x_{i_k,1},\ldots,x_{i_k,n}\mid k\in[s]\}\cup \{x_{1,j_\ell},\ldots,x_{m,j_\ell}\mid \ell\in[m-1-s]\}$.
			Since $A\subseteq V(\Delta_{m,n})\setminus N_{\Delta_{m,n}}(A)$,
			\[
			F_{a,b}\setminus\{x_{a,b}\}\subseteq A\subseteq \{x_{i_k,1},\ldots,x_{i_k,n}\mid k\in[s]\}\cup \{x_{1,j_\ell},\ldots,x_{m,j_\ell}\mid \ell\in[m-1-s]\}.
			\]
			We can assume that  $F_{a,b} = \{x_{1,t_1}, \dots, x_{m,t_m}\}$, where $t_1,\ldots,t_m\in [n]$. 
			
			Consider the subset  $W=\{x_{i,j}\mid i\in[m]\setminus\{i_1,\ldots,i_s,a\}, j\in \{j_1,\ldots,j_{m-1-s}\}\}$ of $V(\Delta_{m,n})$, let $\Delta_{m,n}[W]$ be the induced subcomplex  of  $\Delta_{m,n}$ on the set $W$,
			then there exists a facet of $\Delta_{m,n}[W]$, say $F_{a,b}'$,  which has  the form
			\[
			F_{a,b}'=\{x_{r,t_r}\mid r\in[m]\setminus\{i_1,\ldots,i_s,a\}\text{ and\ }t_r\in \{j_1,\ldots,j_{m-1-s}\}\}.
			\]
			Furthermore,  $F_{a,b}'\subseteq A_1\subset S_1$ and $|F_{a,b}'|=m-1-s$.
			
	By the arbitrariness of the choice of vertex $x_{a,b}\in N_{\Delta_{m,n}}(A)$ and $|N_{\Delta_{m,n}}(A)|\ge 2$, we can choose two different vertices $x_{a_1,b_1}, x_{a_2,b_2}\in N_{\Delta_{m,n}}(A)$. Thus, from the above construction, there exist two the induced subcomplexes $\Delta_{m,n}[W_1]$ and  $\Delta_{m,n}[W_2]$ of  $\Delta_{m,n}$, 
			where
			$W_1=\{x_{i,j}\mid i\in[m]\setminus\{i_1,\ldots,i_s,a_1\}, j\in \{j_1,\ldots,j_{m-1-s}\}\}$ and $W_2=\{x_{i,j}\mid i\in[m]\setminus\{i_1,\ldots,i_s,a_2\}, j\in \{j_1,\ldots,j_{m-1-s}\}\}$. 
			Let  $F_{a_1,b_1}'\in \Facets(\Delta_{m,n}[W_1])$ and $F_{a_2,b_2}'\in \Facets(\Delta_{m,n}[W_2])$ be any two facets from $\Delta_{m,n}[W_1]$ and $\Delta_{m,n}[W_2]$, respectively.
			Then $F_{a_1,b_1}'$ and $F_{a_2,b_2}'$
			have the forms
			\begin{align*}
				F_{a_1,b_1}'&=\{x_{r_1,t_{r_1}}\mid r_1\in[m]\setminus\{i_1,\ldots,i_s,a_1\}\text{ and\ }t_{r_1}\in \{j_1,\ldots,j_{m-1-s}\}\},\\
				F_{a_2,b_2}'&=\{x_{r_2,t_{r_2}}\mid r_2\in[m]\setminus\{i_1,\ldots,i_s,a_2\}\text{ and\ }t_{r_2}\in \{j_1,\ldots,j_{m-1-s}\}\}.
			\end{align*}
			Furthermore, $F_{a_1,b_1}',F_{a_2,b_2}'\subsetneq A_1\subset S_1$ and $|F_{a_1,b_1}'|=|F_{a_2,b_2}'|=m-1-s$. Therefore, for any $\ell\in \{j_1,\ldots,j_{m-1-s}\}$, there exist $r_1,r_2$ such that  $r_1\ne r_2$ and $t_{r_1}=t_{r_2}=\ell$.  By the arbitrariness of the choice of $\ell$, we can conclude that $|A_{1}|\ge 2(m-1-s)$.
			
			By symmetry, we can also conclude that $|A_2| \ge 2s$ and complete the proof. \hfill $\square$
		\end{proof}

		\begin{Theorem}\label{v_facet_ideals}
			Let $\Delta_{m,n}$ be a chessboard complex with $n \geq m \geq 2$. Then
			\[
			\mathrm{v}(\mathcal{F}(\Delta_{m,n}))= 2m-2.
			\]
		\end{Theorem}
		\begin{proof} It suffices to show that $\mathrm{v}(\mathcal{F}(\Delta_{m,n}))\le 2m-2$ by  Lemma \ref{lowbound2}.
			
			Choose $f = \prod\limits_{i=2}^{m} x_{i,i-1}x_{i,i}$, then, for any $k\in[n]$, 
			\[
			x_{1,k}f=
			\begin{cases}
				\biggl(x_{1,k} \left(\prod\limits_{i=2}^{k} x_{i,i-1}\right)\left(\prod\limits_{j=k+1}^{m} x_{j,j}\right)\biggr)\left(\prod\limits_{i=k+1}^{m} x_{i,i-1}\right)\left(\prod\limits_{j=2}^{k} x_{j,j}\right), &\text{if $k\in[m-1]$,}\\[2ex]
				\biggl(x_{1,k} ( \prod\limits_{i=2}^{m} x_{i,i-1}) \biggr)\prod\limits_{j=2}^{m} x_{j,j}, &\text{if $k\in[m,n]$.}\\
			\end{cases}
			\]
			Thus, $x_{1,k}f\in \mathcal{F}(\Delta_{m,n})$, which implies $x_{1,k}\in (\mathcal{F}(\Delta_{m,n}):f)$. Therefore, $\mathfrak{p} \subseteq (\mathcal{F}(\Delta_{m,n}) : f)$, where $\mathfrak{p} = (x_{1,1}, x_{1,2}, \ldots, x_{1,n})$. By the definition of the facet ideal of $\Delta_{m,n}$, $\mathcal{F}(\Delta_{m,n})\subseteq \mathfrak{p}$, thus $(\mathcal{F}(\Delta_{m,n}) : f)\subseteq (\mathfrak{p}:f)\subseteq \mathfrak{p}$. Therefore, $(\mathcal{F}(\Delta_{m,n}) : f)=\mathfrak{p}$, it follows that $\mathrm{v}(\mathcal{F}(\Delta_{m,n}))\le \mathrm{v}_{\mathfrak{p}}(\mathcal{F}(\Delta_{m,n}))\le 2m-2$.
		\end{proof}

		Assume that $t\ge 2$.  By \cite[Theorem 3.3]{JZWZ},  we know that  $\mathcal{F}(\Delta_{m,n})=\bigcap\limits_{|U|+|V|=n+1} \mathfrak{p}_{U,V}$, where $\mathfrak{p}_{U,V}=(x_{i,j}\mid i\in U\subseteq[m],j\in V\subseteq[n])$. The $t$-th symbolic power of $\mathcal{F}(\Delta_{m,n})$ is $\mathcal{F}(\Delta_{m,n})^{(t)}=\bigcap\limits_{|U|+|V|=n+1} \mathfrak{p}_{U,V}^t$. 
	
			\begin{Lemma}\label{symbolic_lowbound3}
			Let $\Delta_{m,n}$ be a chessboard complex with $n \geq m \geq 2$. Then
			\[
			\mathrm{v}(\mathcal{F}(\Delta_{m,n})^{(t)})\ge \Omega_{m,n},
			\]
			where 
			$\Omega_{m,n}=
			\begin{cases}
				mt,  &\text{if $m=n$,}\\
				mt-1,&\text{if $n>m$.}
			\end{cases}
			$
		\end{Lemma}
		\begin{proof}
				Let $\mathcal{F}=\mathcal{F}(\Delta_{m,n})$. Then $\mathcal{F}^{(t)}=\bigcap\limits_{|U|+|V|=n+1} \mathfrak{p}_{U,V}^t$, where $\mathfrak{p}_{U,V}=(x_{i,j}\mid i\in U\subseteq[m],j\in V\subseteq[n])$.  Since  $\mathrm{v}(\mathcal{F}^{(t)})=\min\{\mathrm{v}_{\mathfrak{p}_{U,V}}(\mathcal{F}^{(t)}): |U|+|V|=n+1\}$, it suffices to show that $\mathrm{v}_{\mathfrak{p}_{U,V}}(\mathcal{F}^{(t)})\ge \Omega_{m,n}$ for any  $U\subseteq[m], V\subseteq[n]$ such that $|U|+|V|=n+1$. 
			
			Fix a $\mathfrak{p}_{U,V}$ such that  $U\subseteq[m], V\subseteq[n]$ and  $|U|+|V|=n+1$,  by the definition of $\mathrm{v}_{\mathfrak{p}_{U,V}}(\mathcal{F}^{(t)})$, there exists a monomial $f=\prod\limits_{i=1}^{m}\prod\limits_{j=1}^{n}x_{i,j}^{s_{i,j}}$ of degree $\mathrm{v}_{\mathfrak{p}_{U,V}}(\mathcal{F}^{(t)})$ such that $(\mathcal{F}^{(t)}:f)=\mathfrak{p}_{U,V}$.  We distinguish into the following two cases:

		(1) If $\mathfrak{p}_{U,V}\ne (x_{\ell,j}\mid j\in[n])$ for any   $\ell\in[m]$, then, by  Lemma \ref{symboliccolon},  
		\[
		\deg_{\mathfrak{p}_\ell}(f)=\sum\limits_{k=1}^{n}s_{\ell,k}\ge t,
		\]
		where $\mathfrak{p}_\ell=(x_{\ell,j}\mid j\in[n])$. Thus
			\[
		\deg(f) = \sum_{\ell=1}^m \left( \sum_{k=1}^n s_{\ell, k} \right) \ge mt \ge \Omega_{m,n}.
		\]
		
		(2) If $\mathfrak{p}_{U,V}= (x_{\ell,j}\mid j\in[n])$ for some  $\ell\in[m]$, then we can assume that $\mathfrak{p}_{U,V}=\mathfrak{p}_{1}$.  There are two subcases:
		
		(i)  If $n=m$, then  $\mathfrak{p}_{U,V}\ne \mathfrak{q}_k$ for any $k\in [n]$, where $\mathfrak{q}_k=(x_{i,k}\mid  i\in[m])$.  By  Lemma \ref{symboliccolon},   $\deg_{\mathfrak{q}_k}(f)=\sum\limits_{\ell=1}^{m}s_{\ell,k}\ge t$. Thus $\deg(f)=\sum\limits_{k=1}^{n}(\sum\limits_{\ell=1}^{m}s_{\ell,k})\ge mt\ge \Omega_{m,n}$. 
		
		(ii) If $n>m$, then   $\deg_{\mathfrak{p}_{U,V}}(f)=\deg_{\mathfrak{p}_{1}}(f)=\sum\limits_{k=1}^{n}s_{1,k}=t-1$ and $\deg_{\mathfrak{p}_{\ell}}(f)=\sum\limits_{k=1}^{n}s_{\ell,k}\ge t$ for any $\ell\in[2,m]$. Thus
		\[
		\deg(f)=\sum\limits_{k=1}^{n}s_{1,k}+\sum\limits_{\ell=2}^{m}(\sum\limits_{k=1}^{n}s_{\ell,k})\ge (t-1)+(m-1)t=mt-1\ge  \Omega_{m,n}. \qedhere
		\]
\end{proof}

			\begin{Theorem}\label{symbolicpower2}
				Let $\Delta_{m,n}$ be a chessboard complex with $n \geq m \geq 2$. Then 
				\[
				\mathrm{v}(\mathcal{F}(\Delta_{m,n})^{(t)})= \Omega_{m,n}, 
				\]
				where $\Omega_{m,n}$ is defined as in Lemma \ref{symbolic_lowbound3}. 
			\end{Theorem}
				\begin{proof}
					Let $\mathcal{F}=\mathcal{F}(\Delta_{m,n})$. It suffices to show that   $\mathrm{v}(\mathcal{F}^{(t)})\le \Omega_{m,n}$  by Lemma \ref{symbolic_lowbound3}. 
				
				Choose
				\[
				f=
				\begin{cases}
					x_{1,1}^{t-1}x_{2,m}\prod\limits_{i=2}^{m}x_{i,i-1}x_{i,i}^{t-1},&\text{if $m=n$,}\\[4pt]
					x_{1,m+1}^{t-1}\prod\limits_{i=2}^{m}x_{i,i-1}x_{i,i}^{t-1},    &\text{if $m<n$.}
				\end{cases}
				\]
				Then $\deg(f)=\Omega_{m,n}$.

				We will prove that $(\mathcal{F}^{(t)}:f)=\mathfrak p_{U_0,V_0}$, where $U_0=\{1\}$ and $V_0=[n]$. By  Lemma \ref{symboliccolon}, this is equivalent to proving that $\deg_{\mathfrak p_{U_0,V_0}}(f)=t-1$ and $\deg_{\mathfrak{p}_{U,V}}(f)\ge t$ for any $U\subseteq[m], V\subseteq[n]$ such that $|U|+|V|=n+1$ and $U\neq \{1\}$.  
				
				First, we compute  $\deg_{\mathfrak p_{U_0,V_0}}(f)$.
				By the choice of $f$, 
				\[
				\supp(f)\cap\supp(\mathfrak p_{U_0,V_0})=\begin{cases}
					\{x_{1,1}\},&\text{if $m=n$,}\\ 
					\{x_{1,m+1}\},    &\text{if $m<n$.}
				\end{cases}
				\]
				Thus 
				\begin{align*}
					\deg_{\mathfrak p_{U_0,V_0}}(f)&=\begin{cases}
						\deg_{x_{1,1}}(f),&\text{if $m=n$,}\\ 
						\deg_{x_{1,m+1}}(f),    &\text{if $m<n$.}
					\end{cases}\\
					&=t-1.
				\end{align*}
				Next,  we will prove that $\deg_{\mathfrak{p}_{U,V}}(f)\ge t$ for any  $U\subseteq[m], V\subseteq[n]$ such that $|U|+|V|=n+1$ and $U\neq \{1\}$.  
				
				Let $\mathfrak{q}=\mathfrak{p}_{U,V}$.  We consider the following two cases:
				
				(1)  If one of the following conditions holds: (i) $m = n$ and $2 \in U$, or (ii) $1 \notin U$ for any $n \ge m$, then
				\[
				\deg_{\mathfrak q}(f)=\sum\limits_{i\in U}\sum\limits_{j\in V} \deg_{x_{i,j}}(f)=\sum\limits_{i\in U}\sum\limits_{j\in[n]} \deg_{x_{i,j}}(f)-\sum\limits_{i\in U}\sum\limits_{j\in[n]\setminus V}\deg\limits_{x_{i,j}}(f).
				\]
				To prove $\deg_{\mathfrak{q}}(f)\ge t$, we will  prove 
				\[
				\sum\limits_{i\in U}\sum\limits_{j\in[n]} \deg_{x_{i,j}}(f)\ge |U|t
				\quad\text{and}\quad
				\sum\limits_{i\in U}\sum\limits_{j\in[n]\setminus V}\deg\limits_{x_{i,j}}(f)\le (|U|-1)t.
				\]
				By the expression for $f$, it is easy to see  $\sum\limits_{i\in U}\sum\limits_{j\in[n]} \deg_{x_{i,j}}(f) \ge |U|t$ and  $\sum\limits_{i=1}^m \deg_{x_{i,\ell}}(f) \le t$ for each $\ell \in [n]$. Since $|U|+|V|=n+1$,  $|[n]\setminus V|=n-|V|=|U|-1$. Thus
				\[
				\sum\limits_{i\in U}\sum\limits_{j\in[n]\setminus V}\deg\limits_{x_{i,j}}(f)\le\sum\limits_{i\in [m]}\sum\limits_{j\in[n]\setminus V}\deg\limits_{x_{i,j}}(f)=\sum\limits_{j\in[n]\setminus V}\bigl(\sum\limits_{i\in [m]}\deg\limits_{x_{i,j}}(f)\bigr)\le (|U|-1)t. 
				\]

			(2) If one of the following conditions holds: (i) $m=n$, $1 \in U$ and $2 \notin U$, or (ii) $n>m$  and $1 \in U$, then $|U|\ge 2$, 
			since $U\neq \{1\}$. Thus we set 
			\[
			W=U\setminus\{1\}, \quad j_0=
			\begin{cases}
				1,  &\text{if $m=n$,}\\
				m+1,&\text{if $n>m$,}
			\end{cases} 
			\quad\text{and}\quad  g=x_{1,j_0}^{t-1}\prod\limits_{i\in W}x_{i,i-1}x_{i,i}^{t-1}.
			\]
			Thus  we can obtain that $\deg(g)=|U|t-1$ and that the first index of all variables in the expression for $g$ is taken from $U$.
			Furthermore,  $f=gh$, where 
			\[
			h=
			\begin{cases}
				x_{2,m}\prod\limits_{i\in[2,m]\setminus W}x_{i,i-1}x_{i,i}^{t-1},&\text{if $m=n$,}\\[4pt]
				\prod\limits_{i\in[2,m]\setminus W}x_{i,i-1}x_{i,i}^{t-1},    &\text{if $m<n$.}
			\end{cases}
			\]
			
			It  is clear that $\deg_{\mathfrak q}(h)=0$ and $\deg_{\mathfrak q}(f)=\deg_{\mathfrak q}(g)+\deg_{\mathfrak q}(h)=\deg_{\mathfrak q}(g)$.

				To compute $\deg_{\mathfrak{q}}(g)$, we need to consider the following two subcases:
				
				(I) If $j_0 \notin V$, then $j_0 \in [n] \setminus V$. Note that
				\begin{align*}
				\deg_{\mathfrak{q}}(g) &= \sum_{i \in U} \sum_{j \in V} \deg_{x_{i,j}}(g) = \sum_{i \in U} \sum_{j \in [n]} \deg_{x_{i,j}}(g) - \sum_{i \in U} \sum_{j \in [n] \setminus V} \deg_{x_{i,j}}(g) \\
					&= \deg(g) - \sum_{i \in U} \sum_{j \in [n] \setminus V} \deg_{x_{i,j}}(g).
				\end{align*}
				
				Furthermore,
				\[
				\sum_{i \in U} \sum_{j \in [n] \setminus V} \deg_{x_{i,j}}(g) = \sum_{i \in U} \deg_{x_{i,j_0}}(g) + \sum_{i \in U} \sum_{j \in [n] \setminus (V \cup \{j_0\})} \deg_{x_{i,j}}(g).
				\]
				
				Therefore, if we prove that
				\[
				\sum_{i \in U} \deg_{x_{i,j_0}}(g) = t - 1 \quad \text{and} \quad \sum_{i \in U} \sum_{j \in [n] \setminus (V \cup \{j_0\})} \deg_{x_{i,j}}(g) \le (|U| - 2)t,
				\]
				then we can deduce that $\deg_{\mathfrak{q}}(g)  \ge t$.
				
				It is easy to see that   $\sum\limits_{i\in U}\deg\limits_{x_{i,j_0}}(g)=\deg\limits_{x_{1,j_0}}(g)=t-1$,  since $\supp(g)\cap \{x_{i,j_0}\mid i\in U\}=\{x_{1,j_0}\}$.  
				
				From the expression of $g$, we can obtain  $\sum\limits_{i\in U} \deg_{x_{i,\ell}}(g) \le t$ for every $\ell \in [n]$. 
				Thus 
					\[
				\sum_{i \in U} \sum_{j \in [n] \setminus (V \cup \{j_0\})} \deg_{x_{i,j}}(g) \le (|U| - 2)t,
				\]
since  $|[n]\setminus (V\cup\{j_0\})|=n-|V|-1=|U|-2$. 
				
	\medskip
	(II) If $j_0\in V$, then  
	\begin{align*}
		\deg_{\mathfrak q}(g)&=\sum\limits_{i\in U}\sum\limits_{j\in V}\deg\limits_{x_{i,j}}(g)=\sum\limits_{j\in V}\deg\limits_{x_{1,j}}(g)+\sum\limits_{i\in W}\sum\limits_{j\in V}\deg\limits_{x_{i,j}}(g)\\
		&=\deg_{x_{1,j_0}}(g)+\sum\limits_{i\in W}\sum\limits_{j\in V}\deg\limits_{x_{i,j}}(\prod\limits_{k\in W}x_{k,k-1}x_{k,k}^{t-1})\\
		&=(t-1)+\sum\limits_{i\in W}\sum\limits_{j\in V}\deg\limits_{x_{i,j}}(\prod\limits_{k\in W}x_{k,k-1}x_{k,k}^{t-1}).
	\end{align*}
	To prove that $	\deg_{\mathfrak{q}}(g) \ge t$,  it suffices to show that  $x_{i,j}\mid \prod\limits_{k\in W}x_{k,k-1}x_{k,k}^{t-1}$  for some $i\in W$
	and $j\in V$.
	We divide in the following  two cases:
	
	(i) If $W\cap V\neq \emptyset$, then there is a $\gamma\in W\cap V$ such that $x_{\gamma,\gamma}\mid \prod\limits_{k\in W}x_{k,k-1}x_{k,k}^{t-1}$.
	
	(ii) If $W\cap V = \emptyset$, then $W=[n]\setminus V$,  since $|W|=|[n]\setminus V|=|U|-1$ .
	Let $\delta=\min\{i\mid i\in W\}$, then $\delta\ge 2$, since $1\notin W$. Furthermore, $\delta-1\notin W=[n]\setminus V$, which implies $\delta - 1 \in V$. Therefore, 
	$x_{\delta,\delta-1}\mid \prod\limits_{k\in W}x_{k,k-1}x_{k,k}^{t-1}$
	with $\delta \in W$ and $\delta - 1 \in V$, which implies $\deg_{\mathfrak{q}}(g) \ge t$ and completes the proof.
			\end{proof}
		
\subsection{Ordinary powers}

In the following, we  will show that   $\mathrm{v}(\mathcal{F}(\Delta_{m,n})^t)$ is linear for $t\gg 0$.

\begin{Lemma}\label{F^t}
	Let  $t$ be a positive integer and $\Delta_{m,m}$  a chessboard complex with $m\ge 2$. Assume that      $f=\prod\limits_{i=1}^{m}\prod\limits_{j=1}^{m}x_{i,j}^{s_{i,j}}$ is a monomial   such that $\deg(f)=mt$. Then $f\in \mathcal{F}(\Delta_{m,m})^t$ if and only if $\sum\limits_{i=1}^{m}s_{i,k}=t$ and $\sum\limits_{j=1}^{m}s_{k,j}=t$ for any $k\in[m]$.
\end{Lemma}
\begin{proof}
	Let $f\in \mathcal{F}(\Delta_{m,m})^t$, then $f=u_1u_2\cdots u_t$, since $\deg(f)=mt$. For each $u_i$, we can write $u_i$ as 
	$u_i=x_{1,i_1}\cdots x_{m,i_m}$, where  $i_1,\ldots,i_m$ are pairwise distinct elements of $[m]$. Therefore, $\sum\limits_{j=1}^{m}s_{k,j}=t$ for any $k\in[m]$.
	By  symmetry, we also have $\sum\limits_{i=1}^{m}s_{i,k}=t$ for any $k\in[m]$. Conversely, if $\sum\limits_{i=1}^{m}s_{i,k}=t$ and $\sum\limits_{j=1}^{m}s_{k,j}=t$ for any $k\in[m]$, then by \cite[Theorem B]{K}, we can deduce that $f\in \mathcal{F}(\Delta_{m,m})^t$.
\end{proof}

\begin{Theorem}\label{n=m}
	Let $m\ge 2$ and $t$ be two integers such that  $t\ge m-1$, and let  $\Delta_{m,m}$ be a chessboard complex. Then  
	\[
	\mathrm{v}(\mathcal{F}(\Delta_{m,m})^t)=mt.
	\]
\end{Theorem}
	\begin{proof}
	Let $\mathcal{F}=\mathcal{F}(\Delta_{m,m})$. If $t=1$, then   $\mathrm{v}(\mathcal{F})=m$ by Theorem \ref{v_facet_ideals}. In the following, we assume that $t\ge \max\{2,m-1\}$.  We first prove that $\mathrm{v}(\mathcal{F}^t)\ge  mt$.
	
	By Lemma \ref{lowbound},  $\mathrm{v}(\mathcal{F}^t)\ge mt-1$. If $\mathrm{v}(\mathcal{F}^t)=mt-1$, then there exists a monomial $f$ of degree $mt-1$ and a $\mathfrak{p}\in\Ass(\mathcal{F}^t)$ such that $(\mathcal{F}^t:f)=\mathfrak{p}$.
	Let $f=\prod\limits_{i=1}^{m}\prod\limits_{j=1}^{m}x_{i,j}^{s_{i,j}}$ such that $\sum\limits_{i=1}^{m}\sum\limits_{j=1}^{m}s_{i,j}= mt-1$.

	Claim: $\mathfrak p$ is generated by exactly one variable. 
	In fact, since  $\mathfrak{p}=(\mathcal{F}^t:f)$ and $\deg(f)=mt-1$, for any $x_{a,b}\in \mathfrak{p}$, we have  $x_{a,b}f\in \mathcal{F}^t$ and  $\deg(x_{a,b}f)=mt$.
	It follows from Lemma \ref{F^t} that  $1+\sum\limits_{i=1}^{m}s_{i,b}=t$, $\sum\limits_{i=1}^{m}s_{i,k}=t$ for any $k\in[m]\setminus\{b\}$, 
	$1+\sum\limits_{j=1}^{m}s_{a,j}=t$ and $\sum\limits_{j=1}^{m}s_{\ell,j}=t$ for any $\ell\in[m]\setminus\{a\}$.
	Suppose for contradiction that $\mathfrak p$ is generated by at least two variables, then there exist  two distinct vertices in $\mathfrak p$, say $x_{a_1,b_1}$ and  $x_{a_2,b_2}$.  Without loss of generality, we assume that  $a_1\ne a_2$. Then $a_2\in [m]\setminus\{a_1\}$. From the discussion above, it follows that 
	$1+\sum\limits_{j=1}^{m}s_{a_1,j}=t$ and $\sum\limits_{j=1}^{m}s_{a_2,j}=t$,  a contradiction. 
	
	Let $\mathfrak{p}=(x_{\alpha,\beta})$ for some $\alpha,\beta\in[m]$.  Applying the discussion above once again, we can conclude that
	$1+\sum\limits_{i=1}^{m}s_{i,\beta}=t$,  $\sum\limits_{i=1}^{m}s_{i,k}=t$ for every $k\in[m]\setminus\{\beta\}$,  $1+\sum\limits_{j=1}^{m}s_{\alpha,j}=t$ and  $\sum\limits_{j=1}^{m}s_{\ell,j}=t$ for every $\ell\in[m]\setminus\{\alpha\}$. Then
	\begin{align*}	\sum\limits_{i\in[m]\setminus\{\alpha\}}\sum\limits_{j\in[m]\setminus\{\beta\}}s_{i,j}&=\sum\limits_{i=1}^{m}\sum\limits_{j=1}^{m}s_{i,j}-\sum\limits_{i=1}^{m}s_{i,\beta}-\sum\limits_{j=1}^{m}s_{\alpha,j}+s_{\alpha,\beta}\\
		&=(mt-1)-(t-1)-(t-1)+s_{\alpha,\beta}\ge 1.
	\end{align*}
	Therefore, there exist $u\in[m]\setminus\{\alpha\}$ and $v\in[m]\setminus\{\beta\}$ such that $s_{u,v}\ge 1$. Thus   by Lemma \ref{F^t}, $x_{\alpha,v}x_{u,\beta}f/x_{u,v}\in \mathcal{F}^t$, which implies   $x_{\alpha,v}x_{u,\beta}\in (\mathcal{F}^t:f)=\mathfrak{p}$, which  is a contradiction. 
	Therefore, we can conclude that $\mathrm{v}(\mathcal{F}^t)\ge  mt$.
	
	In the following, we will show  that $\mathrm{v}(\mathcal{F}^t)\le mt$.

	Choose $g =x_{1,1}^{t-1}x_{2,m}^{t-m+2} (\prod\limits_{i=1}^{m-1} x_{2,i})(\prod\limits_{j=3}^{m}x_{j,j-1}^{t-1} x_{j,m})$. Then 
	$x_{2,r}\in \supp(g)$ for any $r\in[m]$.  By Lemma \ref{F^t},  $x_{1,r}g/x_{2,r}\in \mathcal{F}^t$,  implying $x_{1,r}\in(\mathcal{F}^t:g)$. Therefore, $ (x_{1,1}, x_{1,2}, \ldots, x_{1,m})\subseteq(\mathcal{F}^t : g)$. Let  $\mathfrak{p} = (x_{1,1}, x_{1,2}, \ldots, x_{1,m})$, then,  from the proof of Theorem \ref{v_facet_ideals},  we can obtain $\mathfrak{p} \in \Ass(\mathcal{F})$. 
	Note that $\supp(\mathfrak{p})\cap\supp(g)=\{x_{1,1}\}$,  $\deg_\mathfrak{p}(g)=\deg_{x_{1,1}}(g)=t-1$.   By Lemma \ref{colon},  $(\mathcal{F}^t : g) = \mathfrak{p}$ and  it follows that  $\mathrm{v}(\mathcal{F}^t)\le 	\deg(g)= mt$.
\end{proof}

	\begin{Theorem}\label{2m-1}
		Let  $m \geq 2$ and $n$ be two integers such that $n\ge 2m-1$.  Let
		$\Delta_{m,n}$ be a chessboard complex. Then 
		\[
		\mathrm{v}(\mathcal{F}(\Delta_{m,n})^t)=mt-1, \text{\  for all \ } t\ge 2.
		\]
	\end{Theorem}
	
	\begin{proof}
		Let $ \mathcal{F}=\mathcal{F}(\Delta_{m,n})$. Then, by Lemma \ref{lowbound},   $\mathrm{v}( \mathcal{F}^t)\ge mt-1$. We will prove $\mathrm{v}( \mathcal{F}^t)\le mt-1$.
		
		Since $n\ge 2m-1$, we can choose  $f = x_{1,1}^{t-1}(\prod\limits_{i=2}^{m} x_{i,2i-2}^{t-1}x_{i,2i-1})$. Then, for any $\alpha\in [n]$,  
		\[
	x_{1,\alpha}f=
	\begin{cases}
		h(x_{1,\alpha}\prod\limits_{j=2}^{m}x_{j,2j-2})(x_{1,1}\prod\limits_{k=2}^{m} x_{k,2k-1}), &\text{if $\alpha\in[2m-1]$ is odd,}\\
		h(x_{1,\alpha}\prod\limits_{j=2}^{m}x_{j,2j-1})(x_{1,1}\prod\limits_{\ell=2}^{m} x_{\ell,2\ell-2}),  &\text{otherwise,}
	\end{cases}
	\] 
	where  $h=(x_{1,1}\prod\limits_{i=2}^{m}x_{i,2i-2})^{t-2}$.
		Thus $x_{1,\alpha}f\in \mathcal{F}^t$,  so  $x_{1,\alpha}\in (\mathcal{F}^t:f)$.  From the arbitrariness of $\alpha$, we can conclude that
		$(x_{1,1}, x_{1,2}, \ldots, x_{1,n})\subseteq(\mathcal{F}^t : f)$.  Let  $\mathfrak{p} = (x_{1,1}, x_{1,2}, \ldots, x_{1,n})$, then 
		$\mathfrak{p}\subseteq(\mathcal{F}^t : f)$ and $\mathfrak{p} \in \Ass( \mathcal{F})$. 
		Note that $\supp(\mathfrak{p})\cap\supp(f)=\{x_{1,1}\}$,   $\deg_\mathfrak{p}(f)=\deg_{x_{1,1}}(f)=t-1$.  By Lemma \ref{colon},  $( \mathcal{F}^t : f) = \mathfrak{p}$ and  it follows that   $\mathrm{v}( \mathcal{F}^t)\le\deg(f)= mt-1$. 
	\end{proof}
	
\begin{Theorem}\label{n>m}
	Let  $m \geq 2$ and $n$ be two integers such that $m< n\le 2m-2$, and  let
 $\Delta_{m,n}$ be a chessboard complex. Then  
 \[
 \mathrm{v}(\mathcal{F}(\Delta_{m,n})^t)=mt-1, \text{\  for all \ } t\ge m.
 \]
\end{Theorem}
\begin{proof}
	Let $\mathcal{F}=\mathcal{F}(\Delta_{m,n})$. Then, by Lemma \ref{lowbound},   $\mathrm{v}( \mathcal{F}^t)\ge mt-1$. We will prove $\mathrm{v}( \mathcal{F}^t)\le mt-1$.
	
	Choose $f =(\prod\limits_{i=1}^{m} x_{i,i}^{t-1})(\prod\limits_{j=2}^{m} x_{j,m+1})$. Then,  for any $\ell\in [n]$,
	\[
	x_{1,\ell}f=
	\begin{cases}
		h(x_{1,\ell}x_{\ell,m+1}\prod\limits_{i\in[2,m]\setminus\{\ell\}} x_{i,i})(\prod\limits_{j=1}^{m} x_{j,j}), &\text{if $\ell\in [2,m]$,}\\[1ex]
	 h(x_{1,\ell}\prod\limits_{i=2}^{m} x_{i,i}),, &\text{if $\ell\in\{1\}\cup[m+1,n]$,}
	\end{cases}
	\] 
	where $h=(\prod\limits_{j=1}^{m} x_{j,j}^{t-m})\bigl(\prod\limits_{k\in[2,m]\setminus\{\ell\}}(x_{1,1}x_{k,m+1}\prod\limits_{s\in[2,m]\setminus\{k\}} x_{s,s})\bigr)$. 
	Thus, $x_{1,\ell}f\in \mathcal{F}^t$,   implying $x_{1,\ell}\in (\mathcal{F}^t:f)$. That is, $\mathfrak{p}\subseteq(\mathcal{F}^t:f)$, where $\mathfrak{p} = (x_{1,1}, x_{1,2}, \ldots, x_{1,n})$.   From the proof of Theorem \ref{v_facet_ideals}, we know that $\mathfrak{p} \in \Ass(\mathcal{F})$. 
Since $\supp(\mathfrak{p})\cap\supp(f)=\{x_{1,1}\}$,  $\deg_\mathfrak{p}(f)=\deg_{x_{1,1}}(f)=t-1$. 
By Lemma \ref{colon},  $( \mathcal{F}^t : f) = \mathfrak{p}$ and  it follows that   $\mathrm{v}( \mathcal{F}^t)\le \deg(f)=mt-1$. 
\end{proof}

\medskip

\medskip

\hspace{-6mm} {\bf Acknowledgment: }
This research is supported by the Natural Science Foundation of Jiangsu Province (No. BK20221353) and the National Natural Science Foundation of China (No.12471246).  The second author carried out part of this work while visiting the International Center for Research in Mathematics and Postgraduate Training (ICRTM), affiliated with the Institute of Mathematics at the Vietnam Academy of Science and Technology, through the ICRTM--TWAS Visiting Program. She expresses her sincere gratitude for the financial support and warm hospitality received during this period.

\medskip
\hspace{-6mm} {\bf Data availability statement}

\vspace{3mm}
\hspace{-6mm}  The data used to support the findings of this study are included within the article.

\medskip
\hspace{-6mm} {\bf Conflict of interest}

\vspace{3mm}
\hspace{-6mm}  The authors declare that they have no competing interests.

\end{document}